\documentclass{article}
\usepackage{amsfonts,color}
\usepackage{amsmath}
\usepackage{epsfig}
\usepackage{lscape}
\usepackage{amssymb}
\usepackage{bbm}

\renewcommand{\epsilon}{\varepsilon}

\allowdisplaybreaks[1]

\makeindex

\newtheorem{theorem}{Theorem}

\newtheorem{corollary}[theorem]{Corollary}

\newenvironment{proof}[1][Proof]{\noindent\textbf{#1.} }{\ \rule{0.5em}{0.5em}}

\def\barr{\begin{array}}

	\def\earr{\end{array}}
\def\bec#1{\begin{equation}\label{#1}}
\def\becn{\begin{equation*}}
\def\endec{\end{equation}}
\def\endecn{\end{equation*}}

\usepackage{footnote}

\makeatletter
\let\@fnsymbol\@arabic
\makeatother

\begin{document}
	\title{Closed-form Saint-Venant solutions in the Koiter theory of shells}
	\author{ Mircea B\^irsan\thanks{Mircea B\^irsan, \ \  Fakult\"{a}t f\"{u}r Mathematik,
			Universit\"{a}t Duisburg-Essen, Thea-Leymann Str. 9, 45127 Essen, Germany; and  Department of Mathematics, Alexandru Ioan Cuza University of Ia\c si,  Blvd.
			Carol I, no. 11, 700506 Ia\c si,
			Romania.;  email: mircea.birsan@uni-due.de}  
	}

\maketitle

\begin{abstract}
In this paper we investigate the deformation of cylindrical linearly elastic shells using the Koiter model. We formulate and solve the relaxed Saint-Venant's problem for thin cylindrical tubes made of isotropic and homogeneous elastic materials. To this aim, we adapt a method established previously in the three-dimensional theory of elasticity. We present a general solution procedure to determine closed-form solutions for the extension, bending, torsion and flexure problems. We remark the analogy and formal resemblance of these solutions to the classical Saint-Venant's solutions for solid cylinders. The special case of circular cylindrical shells is also discussed.
\end{abstract}
\bigskip

\bigskip


\section{Introduction}
\label{intro}

In this paper, we consider the linear theory of elastic shells using the model of Koiter \cite{Koiter60,Koiter70}. The derivation of Koiter's theory of shells starting from the three-dimensional elasticity theory  has been presented in a new perspective in \cite{Steigmann13}. For the mathematical investigation of Koiter's equations for linearly elastic shells we refer to the books of Ciarlet \cite{Ciarlet00,Ciarlet05}. 

Within this framework, we approach the well-known Saint-Venant's problem, which consists in determining the equilibrium deformation of a cylindrical shell subjected to mechanical loads distributed over its end edges. Due to its significance in practice, this problem has been intensively studied both in the three-dimensional elasticity (i.e., for solid cylinders), as well as in the shell theory \cite{Berdichevsky92,Ladeveze04,Reissner72}. Barr\'{e} de Saint-Venant \cite{Saint-Venant1843} have found a set of solutions to the relaxed formulation of the problem, which represent a classical result in the three-dimensional elasticity and are widely used in engineering. 
In our paper, we derive the analogues of the classical Saint-Venant's solutions for the deformation of cylindrical shells (elastic tubes). To this purpose we adapt the method established by Ie\c san \cite{Iesan86,Iesan87} in the context of three-dimensional elasticity for the case of classical shells.

We recapitulate first the Koiter's equations for linearly elastic shells (Section \ref{Sect2}) and formulate the relaxed Saint-Venant's problem for cylindrical shells (Section \ref{Sect3}). Due to the linearity of the theory, we can decompose it into two problems, namely the extension-bending-torsion problem and the flexure problem.
In Section \ref{Sect4} we solve these two problems and determine an exact solution written in closed form. 
In Section \ref{Sect5} we deduce an approximate form of the solution, valid for shells with a very small thickness. We observe that the simplified form of the solution for shells is similar to the classical Saint-Venant's solution for solid cylinders. These results hold for isotropic and homogeneous cylindrical tubes with arbitrary shape of the cross-section. Finally, we apply these general results in Section \ref{Sect6} to the case of circular cylindrical shells and derive the solutions in closed form.

We mention that some of the results obtained here are already known and can be found in the literature, since this is an important problem of the elasticity theory and the Koiter model of shells is classical. However, we present the solutions of the extension, bending, torsion and flexure problems in an unified manner, using a general solution procedure, which is employed in the Koiter shell theory for the first time.
This method to solve the relaxed Saint-Venant's problem can also be extended to the case of anisotropic shells.

\section{Koiter's equations for linear elastic shells}
\label{Sect2}

In order to introduce the necessary notations, we present briefly the basic equations of the linear shell model of Koiter.

Let $ \mathcal{S} $ be the reference midsurface of a linearly elastic shell and let $ (\theta^1,\theta^2) $ denote the curvilinear material coordinates on $ \mathcal{S} $. Further, we designate by  $ \boldsymbol r(\theta^1,\theta^2) $ the position vector of any point on $ \mathcal{S} $ with respect to a fixed origin $ O $;
the covariant base vectors are $ \boldsymbol a_\alpha \,$, the contravariant base vectors  $ \boldsymbol a^\alpha \,$,
the unit normal  $ \boldsymbol n $, the metric tensor is $ a_{\alpha\beta}\, $, and the curvature tensor  $ \kappa_{\alpha\beta} \,$. They are given by the following well-known relations
\begin{equation}\label{eq1}
\begin{array}{c}
\displaystyle\boldsymbol a_\alpha= \frac{\partial \boldsymbol r}{\partial\theta^\alpha}=\boldsymbol r_{,\alpha}\,,\qquad \boldsymbol a^\alpha\cdot \boldsymbol a_\beta  =\delta^\alpha_\beta\,,\qquad \boldsymbol n = \frac{\boldsymbol a_1\times \boldsymbol a_2}{| \boldsymbol a_1\times \boldsymbol a_2|}\,, \vspace{6pt}\\  
a_{\alpha\beta} = \boldsymbol a_\alpha\cdot \boldsymbol a_\beta \,,\qquad \kappa_{\alpha\beta} = -\boldsymbol n_{,\alpha}\cdot \boldsymbol a_\beta \,,
\end{array}
\end{equation}
where the Greek indices take the values $ \{1,2\} $.
According to the usual indicial notation, a subscript comma represents partial differentiation. The summation convention for repeated indices is also used.

If $ \boldsymbol u(\theta^1,\theta^2) $ denotes the displacement field, then the linearized strain tensors are the change of metric tensor $ \epsilon_{\alpha\beta} $ and change of curvature tensor $ \rho_{\alpha\beta}\, $, given by
\begin{equation}\label{eq2}
\epsilon_{\alpha\beta} = \frac12(\boldsymbol a_\alpha\cdot \boldsymbol u_{,\beta} + \boldsymbol a_\beta\cdot \boldsymbol u_{,\alpha} ) ,\qquad
\rho_{\alpha\beta} = \boldsymbol n\cdot \boldsymbol u_{;\alpha\beta} = \boldsymbol n\cdot (\boldsymbol u_{,\alpha\beta}- \Gamma^\lambda_{\alpha\beta} \boldsymbol u_{,\lambda}),
\end{equation}
where $ \Gamma^\lambda_{\alpha\beta} $ are the Christoffel symbols and the subscript semicolon denotes covariant derivative.

\begin{figure}
	\begin{center}
		\includegraphics[scale=0.25]{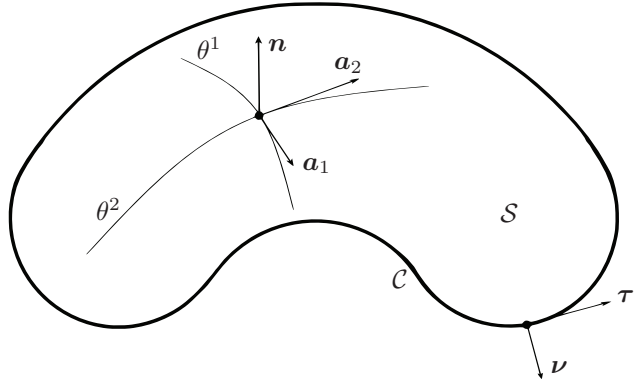}
		\put(-13,30){$ \boldsymbol \tau $} \put(-38,3){$ \boldsymbol\nu $}
		\put(-98,35){$ \mathcal{C} $} \put(-57,60){$ \mathcal{S} $}
		\put(-132,80){$ \boldsymbol a_1 $} \put(-120,118){$ \boldsymbol a_2 $}
		\put(-145,125){$ \boldsymbol n $}
		\put(-172,122){$ \theta^1 $} \put(-210,61){$ \theta^2 $}
		\caption{The reference midsurface $ \mathcal{S} $ of the shell.}
		\label{Fig1}       
	\end{center}
\end{figure}

The equilibrium equations can be written in the form
$$
\boldsymbol T^\alpha{}_{;\alpha} + \boldsymbol g = \boldsymbol 0,
$$
or equivalently,
\begin{equation}\label{eq3}
P^{\beta\alpha}{}_{;\alpha} - \kappa_\alpha^\beta S^\alpha + g^\beta =0,\qquad 
S^\alpha{}_{;\alpha}  + \kappa_{\beta\alpha} P^{\beta\alpha} + g =0,
\end{equation}
where $ \boldsymbol g= g^\alpha\boldsymbol a_\alpha + g\,\boldsymbol n $ is the assigned force per unit area and 
\begin{equation}\label{eq4}
\boldsymbol T^\alpha = P^{\beta\alpha}\boldsymbol a_\beta + S^\alpha\boldsymbol n\,.
\end{equation}
The tangential tractions $ P^{\beta\alpha} $ and transverse shear tractions $ S^\alpha $ are expressed in terms of the stress and couple stress tensors $ N^{\alpha\beta} $ and $ M^{\alpha\beta} $ by the relations
\begin{equation}\label{eq5}
P^{\alpha\beta} = N^{\alpha\beta} + \kappa_{\gamma}^\alpha M^{\gamma\beta}, \qquad  
S^\alpha = -M^{\beta\alpha}{}_{;\beta} =  -M^{\beta\alpha}{}_{,\beta} - \Gamma^\beta_{\gamma\beta}M^{\gamma\alpha} - \Gamma^\alpha_{\gamma\beta}M^{\beta\gamma} .
\end{equation}
For isotropic materials the stress-strain relations have the form 
\begin{equation}\label{eq6}
N^{\alpha\beta} = C\left[\nu a^{\alpha\beta}\epsilon_\gamma^\gamma + (1-\nu)\epsilon^{\alpha\beta} \right],
\qquad
M^{\alpha\beta} = D\left[\nu a^{\alpha\beta}\rho_\gamma^\gamma + (1-\nu)\rho^{\alpha\beta} \right],
\end{equation}
where $ C=\frac{Eh}{1-\nu^2}\, $ is the stretching stiffness and $ D=\frac{Eh^3}{12(1-\nu^2)}\, $ the bending stiffness of the shell; $ E $ is the Young modulus of the material, $ \nu $ the Poisson ratio and $ h $ the thickness of the shell.

Let $ \mathcal{C} $ be the boundary curve of the surface $ \mathcal{S} $, see Figure \ref{Fig1}. According to the Koiter model, we consider the boundary conditions in the form
\begin{equation}\label{eq7}
\boldsymbol T^\alpha \nu_\alpha - \Big( M^{\alpha\beta}\nu_\alpha\tau_\beta\,\boldsymbol n\Big)_{,s} =\boldsymbol f
\qquad\mbox{and}\qquad 
M^{\alpha\beta}\nu_\alpha\nu_\beta\, \boldsymbol n = \boldsymbol c\,,
\end{equation}
where $ \boldsymbol f $ and $ \boldsymbol c $ are the assigned force and moment per unit length of $ \mathcal{C} $. Here, $ \boldsymbol \tau = \tau_\alpha\boldsymbol a^\alpha $ is the unit tangent vector along  $ \mathcal{C} $ and $ \boldsymbol \nu = \nu_\alpha\boldsymbol a^\alpha $ is the unit outer normal to $ \mathcal{C} $ lying in the tangent plane, while $ s $ is the arclength parameter along $ \mathcal{C} $.

On the basis of the equations (\ref{eq3}) and boundary conditions (\ref{eq7}) one can show that the following relations hold in case of equilibrium
\begin{equation}\label{eq8}
\int_{\mathcal{S}} \boldsymbol g \,\mathrm{d}a + 
\int_{\mathcal{C}}  \boldsymbol f \,\mathrm{d}s = \boldsymbol 0,\qquad
\int_{\mathcal{S}} \boldsymbol r\times \boldsymbol g \,\mathrm{d}a + 
\int_{\mathcal{C}} \big(\boldsymbol r\times \boldsymbol f + \boldsymbol \nu\times\boldsymbol c \big) \,\mathrm{d}s = \boldsymbol 0.
\end{equation}

\section{Formulation of the relaxed Saint-Venant's problem for cylindrical shells}
\label{Sect3}

We consider general closed cylindrical shells, i.e. cylindrical thin tubes of arbitrary cross section. The cylindrical midsurface $ \mathcal{S} $ of the shell is referred to an orthogonal Cartesian coordinate frame $ Ox_1x_2x_3 $ with unit vectors $ \boldsymbol e_i $ along the axes $ Ox_i $ (the Latin indices take the values $ \{1,2,3\} $), see Figure \ref{Fig2}. The curvilinear coordinates on $ \mathcal{S} $ are denoted by $ \theta^1=s $, $ \theta^2=z $ and the parametrization of $ \mathcal{S} $  is given by
\begin{equation}\label{eq9}
\boldsymbol r(s,z) = x_\alpha(s)\,\boldsymbol e_\alpha + z\,\boldsymbol e_3\,,\qquad s\in[0,\bar{s}],\; z\in [0,\ell],
\end{equation}
where $ x_\alpha(s) $ are given functions of class $ C^2[0,\bar{s}] $, which describe the form of the cross-section curve. Here, $ s $ is the arclength parameter along the cross-section curves and $ z=x_3 $ is the distance to the coordinate plane $ Ox_1x_2\, $. Thus, $ \bar{s} $ stands for the length of the cross-section curve.

\begin{figure}
	\begin{center}
	\includegraphics[scale=0.3]{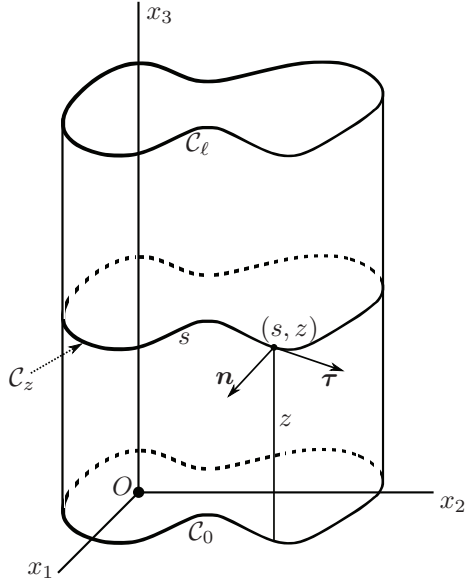}
\put(-12,42){$ x_2 $} \put(-168,17){$ x_1 $} \put(-123,227){$ x_3 $}
\put(-136,47){$ O $} \put(-107,30){$ \mathcal{C}_0 $}
\put(-73,73){$ z $} \put(-97,89){$ \boldsymbol n $} \put(-57,88){$ \boldsymbol \tau $}
\put(-80,107){$ (s,z) $} \put(-111,104){$ s $} 
\put(-175,88){$ \mathcal{C}_z $} \put(-108,177){$ \mathcal{C}_\ell $}
	\caption{The reference configuration of the cylindrical shell.}
	\label{Fig2}       
\end{center}
\end{figure}

We denote by $ \mathcal{C}_z $ the cross-section curve at the distance $ x_3=z $; accordingly, $ \mathcal{C}_0 $ and $ \mathcal{C}_\ell $ are the end edge curves described by $ x_3=0 $ and $ x_3=\ell $, respectively. We note that the curvilinear coordinate system $ \theta^1=s $, $ \theta^2=z $ is orthonormal and we have
\begin{equation}\label{eq10}
\begin{array}{c}
\displaystyle\boldsymbol a_1= \boldsymbol a^1= \frac{\partial \boldsymbol r}{\partial s} = \boldsymbol\tau = x_\alpha^\prime(s)\boldsymbol e_\alpha  \,,\qquad \boldsymbol a_2= \boldsymbol a^2 =\frac{\partial \boldsymbol r}{\partial z} = \boldsymbol e_3\,,
\vspace{6pt}\\ 
\displaystyle\boldsymbol n = \frac{\boldsymbol a_1\times \boldsymbol a_2}{| \boldsymbol a_1\times \boldsymbol a_2|} = e_{\alpha\beta}x_\beta^\prime(s)\boldsymbol e_\alpha \,,
\end{array}
\end{equation}
where $ e_{\alpha\beta} $ is the two-dimensional alternator (i.e., $ e_{12}=-e_{21}=1$,  $ e_{11}=e_{22}=0$); $ \boldsymbol \tau $ and $ \boldsymbol n $ are the unit tangent and normal vectors to the cross-section curve $ \mathcal{C}_z \,$. They satisfy the relations
\begin{equation}\label{eq11}
\boldsymbol\tau^{\prime}(s)= -\frac{1}{R(s)}\,\boldsymbol n(s),\quad 
\boldsymbol n^{\prime}(s)= \frac{1}{R(s)}\,\boldsymbol\tau(s)
\quad \mbox{with}\quad 
\frac{1}{R(s)}= e_{\alpha\beta}x_\alpha^\prime(s)x_\beta^{\prime\prime}(s),
\end{equation}
where $ R(s) $ is the curvature radius of the curve $ \mathcal{C}_z \,$. The covariant and contravariant components of any tensor are equal. We have 
$ a_{\alpha\beta}=\delta_{\alpha\beta} $ (Kronecker symbol) and
\begin{equation}\label{eq12}
\kappa_{11}= -\frac{1}{R(s)}\;,\qquad \kappa_{12}=\kappa_{21}=\kappa_{22}=0,\qquad \Gamma_{\alpha\beta}^\gamma=0.
\end{equation}
In view of $ \theta^1=s $, $ \theta^2=z $, we use the subscripts $ s $ and $ z $ instead of 1 and 2, respectively, for any tensor components. Thus, the strain measures (\ref{eq2}) for cylindrical shells can be written in the form
\begin{equation}\label{eq13}
\begin{array}{c}
\epsilon_{ss}= \boldsymbol u_{,s}\cdot \boldsymbol\tau\,, \qquad 
\epsilon_{sz}= \epsilon_{zs}= \frac12 (\boldsymbol u_{,z}\cdot \boldsymbol \tau + \boldsymbol u_{,s}\cdot\boldsymbol e_3),\qquad 
 \epsilon_{zz}=  \boldsymbol u_{,z}\cdot\boldsymbol e_3\,,
 \vspace{4pt}\\
 \rho_{ss}= \boldsymbol u_{,ss}\cdot \boldsymbol n\,, \qquad 
 \rho_{sz}= \rho_{zs}= \boldsymbol u_{,sz}\cdot \boldsymbol n,\qquad 
 \rho_{zz}=  \boldsymbol u_{,zz}\cdot \boldsymbol n\,.
\end{array}
\end{equation}
The equilibrium equations (\ref{eq3}) in the absence of body forces ($ \boldsymbol g = \boldsymbol 0 $) reduce to
\begin{equation}\label{eq14}
P_{ss,s} + P_{sz,z} + \frac{1}{R}\, S_s =0, \qquad
P_{zs,s} + P_{zz,z} = 0, \qquad
S_{s,s} +  S_{z,z} - \frac{1}{R}\,P_{ss}  =0.
\end{equation}
According to (\ref{eq5}) and (\ref{eq6}) we have the constitutive relations
\begin{equation}\label{eq15}
\begin{array}{c}
\displaystyle
P_{ss}  = N_{ss}- \frac{1}{R}\,M_{ss}\,,\qquad 
P_{sz} = N_{sz}  -\frac{1}{R}\, M_{sz}\, , \qquad
P_{zs}  = N_{zs}\,,
\vspace{4pt}\\
P_{zz} = N_{zz}\,, \qquad 
S_{s} =  -M_{ss,s}-M_{zs,z}\,,\qquad
S_{z}  = -M_{sz,s}-M_{zz,z}
\end{array}
\end{equation}
and
\begin{equation}\label{eq16}
\begin{array}{c}
\displaystyle
N_{ss}  = C(\epsilon_{ss}+ \nu\epsilon_{zz}),\quad 
N_{sz} = N_{zs} = C(1-\nu) \epsilon_{sz}\, , \quad
N_{zz}  = C(\nu\epsilon_{ss}+ \epsilon_{zz}),
\vspace{4pt}\\
M_{ss}  = D(\rho_{ss}+ \nu\rho_{zz}),\quad 
M_{sz} = M_{zs} = D(1-\nu) \rho_{sz}\, , \quad
M_{zz}  = D(\nu\rho_{ss}+ \rho_{zz}).
\end{array}
\end{equation}
The boundary of $ \mathcal{S} $ consists of the end edge curves $ \mathcal{C}_0 $ and $ \mathcal{C}_\ell \,$. Thus, the boundary conditions (\ref{eq7}) become
\begin{equation}\label{eq17}
\begin{array}{r}
\displaystyle
-\big( P_{sz}\boldsymbol \tau + P_{zz}\boldsymbol e_3 + S_z \boldsymbol n \big) + \big( M_{zs} \boldsymbol n \big)_{,s} = \boldsymbol f\qquad \mbox{and}\qquad
M_{zz} \boldsymbol n = \boldsymbol c\quad\mbox{on}\quad \mathcal{C}_0\,,
\vspace{4pt}\\
\displaystyle
\big( P_{sz}\boldsymbol \tau + P_{zz}\boldsymbol e_3 + S_z \boldsymbol n \big) - \big( M_{zs} \boldsymbol n \big)_{,s} = \boldsymbol f\qquad \mbox{and}\qquad 
M_{zz} \boldsymbol n = \boldsymbol c\quad\mbox{on}\quad \mathcal{C}_\ell\,.
\end{array}
\end{equation}

The problem of Saint-Venant consists in determining the equilibrium of cylindrical shells subjected to given loads $ \boldsymbol f $, $ \boldsymbol c $ on the end edges $ \mathcal{C}_0 $ and $ \mathcal{C}_\ell \,$, i.e. solving the equations (\ref{eq13})--(\ref{eq16}) with boundary conditions (\ref{eq17}). In the relaxed formulation of Saint-Venant's problem we replace the pointwise assignment of loads on the end edges by prescribing the resultant force and resultant moment of the given loads.

If we denote by $ \boldsymbol{\mathcal{R}}^0 = \mathcal{R}^0_i\boldsymbol e_i $ the resultant force and by $ \boldsymbol{\mathcal{M}}^0 = \mathcal{M}^0_i\boldsymbol e_i $ the resultant moment about $ O $ on the end edge $ \mathcal{C}_0 \,$, then we have
\begin{equation}\label{eq18}
\int_{\mathcal{C}_0} \boldsymbol f \,\mathrm{d}s = \boldsymbol{\mathcal{R}}^0,\qquad 
\int_{\mathcal{C}_0} (\boldsymbol r\times\boldsymbol f + \boldsymbol\nu\times \boldsymbol c) \,\mathrm{d}s = \boldsymbol{\mathcal{M}}^0.
\end{equation}
Inserting the expressions (\ref{eq17})$ _1 $ into (\ref{eq18}) and making some simplifications we obtain the following form of the resultant boundary conditions on $ \mathcal{C}_0 $
\begin{equation}\label{eq19}
\begin{array}{l}
\displaystyle\int_{\mathcal{C}_0} \big(P_{sz}\boldsymbol \tau + P_{zz}\boldsymbol e_3 + S_z\boldsymbol n \big) \,\mathrm{d}s = -\boldsymbol{\mathcal{R}}^0,  \vspace{6pt}\\
\displaystyle\int_{\mathcal{C}_0} \big[\boldsymbol r\times\big(P_{sz}\boldsymbol \tau + P_{zz}\boldsymbol e_3 + S_z\boldsymbol n \big) - M_{sz}\boldsymbol e_3 + M_{zz}\boldsymbol \tau
\big] \,\mathrm{d}s = -\boldsymbol{\mathcal{M}}^0.
\end{array}
\end{equation}
Since the boundary is $ \partial\mathcal S =\mathcal{C}_0\cup \mathcal{C}_\ell \, $, by virtue of the equilibrium conditions (\ref{eq8}) and the relations (\ref{eq18}) we deduce that the corresponding boundary conditions on the end edge $ \mathcal{C}_\ell $ 
are automatically satisfied and do not represent an additional condition to be imposed on the solution. For the sake of completeness, we record the end edge conditions on $ \mathcal{C}_\ell $ in the form
\begin{equation}\label{eq18bis}
\int_{\mathcal{C}_\ell} \boldsymbol f \,\mathrm{d}s = -\boldsymbol{\mathcal{R}}^0,\qquad 
\int_{\mathcal{C}_\ell} (\boldsymbol r\times\boldsymbol f + \boldsymbol\nu\times \boldsymbol c) \,\mathrm{d}s = -\boldsymbol{\mathcal{M}}^0.
\end{equation}

To conclude, the relaxed Saint-Venant's problem is: find the equilibrium displacement field $ \boldsymbol u $ which satisfies the governing equations (\ref{eq13})--(\ref{eq16}) and the resultant boundary conditions  (\ref{eq19}) on the end edge $ \mathcal{C}_0 \,$.

\section{The exact solution in closed form}
\label{Sect4}

In this section we determine an exact solution to the relaxed Saint-Venant's problem. To this aim, we employ the method elaborated by Ie\c san \cite{Iesan86,Iesan87} in the classical elasticity. 

For any displacement field $ \boldsymbol u $ we define the linear functionals 
$ \boldsymbol{\mathcal{R}}(\boldsymbol u) = \mathcal{R}_i(\boldsymbol u)\boldsymbol e_i $ and $ \boldsymbol{\mathcal{M}}(\boldsymbol u) = \mathcal{M}_i(\boldsymbol u)\boldsymbol e_i $ by the relations
\begin{equation}\label{eq21}
\begin{array}{l}
\boldsymbol{\mathcal{R}}(\boldsymbol u)= - \displaystyle\int_{\mathcal{C}_0} \big[P_{sz}(\boldsymbol u)\boldsymbol \tau + P_{zz}(\boldsymbol u)\boldsymbol e_3 + S_z(\boldsymbol u)\boldsymbol n  \big] \mathrm{d}s  ,  
\vspace{6pt}\\
\boldsymbol{\mathcal{M}}(\boldsymbol u) = - \displaystyle\int_{\mathcal{C}_0} \Big[\boldsymbol r\times\big(P_{sz}(\boldsymbol u)\boldsymbol \tau + P_{zz}(\boldsymbol u)\boldsymbol e_3 + S_z(\boldsymbol u)\boldsymbol n \big) - M_{sz}(\boldsymbol u)\boldsymbol e_3 + M_{zz}(\boldsymbol u)\boldsymbol \tau \Big] \mathrm{d}s ,
\end{array}
\end{equation}
which represent the corresponding resultant force and moment on the end edge $ \mathcal{C}_0 \,$. Then, the end edge conditions (\ref{eq19}) can be written in the compact form
\begin{equation}\label{eq22}
\boldsymbol{\mathcal{R}}(\boldsymbol u)=\boldsymbol{\mathcal{R}}^0,\qquad \boldsymbol{\mathcal{M}}(\boldsymbol u) = \boldsymbol{\mathcal{M}}^0. 
\end{equation}
The following result will be useful in the sequel.

\begin{theorem}\label{Th1}
	Let  $ \boldsymbol u $  be a displacement field which satisfies the governing equations (\ref{eq13})--(\ref{eq16}). If the field $ \frac{\partial\boldsymbol u  }{\partial z} \,$ has the regularity $ C^1(\bar{\mathcal S}) \cap C^2({\mathcal S}) $, then $ \frac{\partial\boldsymbol u  }{\partial z} \,$ is also a solution of the governing equations (\ref{eq13})--(\ref{eq16}) and it satisfies
	\begin{equation}\label{eq23}
	\boldsymbol{\mathcal{R}}\left(\frac{\partial \boldsymbol u}{\partial z}\right)=\boldsymbol{0},\qquad \boldsymbol{\mathcal{M}}\left(\frac{\partial \boldsymbol u}{\partial z}\right) = e_{\alpha\beta}\mathcal{R}_\beta(\boldsymbol u)\,\boldsymbol{e}_\alpha\,.
	\end{equation}
\end{theorem}
\begin{proof}
	In view of the linearity of the equations and the fact that all the coefficients in the relations (\ref{eq13})--(\ref{eq16}) are independent of $ z $, we deduce by differentiation that $ \frac{\partial\boldsymbol u  }{\partial z} \,$ also satisfies the field equations (\ref{eq13})--(\ref{eq16}).
	
	Using the equilibrium equations (\ref{eq14}) and the relations (\ref{eq11}) we get
$$
\begin{array}{rcl}
\displaystyle
\boldsymbol{\mathcal{R}}\left(\frac{\partial \boldsymbol u}{\partial z}\right) & = & - \displaystyle\int_{\mathcal{C}_0} \frac{\partial }{\partial z} \big[P_{sz}(\boldsymbol u)\boldsymbol \tau + P_{zz}(\boldsymbol u)\boldsymbol e_3 + S_z(\boldsymbol u)\boldsymbol n  \big] \,\mathrm{d}s
\vspace{4pt}\\
 & = & \displaystyle\int_{\mathcal{C}_0}  \Big[\big( P_{ss,s} +  \frac{1}{R}\, S_s \big)\boldsymbol \tau +  P_{zs,s} \,\boldsymbol e_3 +\big( S_{s,s} - \frac{1}{R}\,P_{ss} \big)\boldsymbol n  \Big] \,\mathrm{d}s
 \vspace{4pt}\\
 & = &   \displaystyle\int_{\mathcal{C}_0} \frac{\partial }{\partial s} \big(P_{ss}\boldsymbol \tau + P_{zs}\boldsymbol e_3 + S_s\boldsymbol n \big) \,\mathrm{d}s = \boldsymbol 0.
\end{array}
$$
	Analogously, we obtain
$$
\begin{array}{l}
\displaystyle
{\mathcal{M}}_3\left(\frac{\partial \boldsymbol u}{\partial z}\right)   =   \displaystyle\int_{\mathcal{C}_0}  \big[e_{\alpha\beta}x_\alpha^\prime x_\beta P_{sz,z}(\boldsymbol u) + x_\alpha x_\alpha^\prime S_{z,z}(\boldsymbol u)+ M_{sz,z}(\boldsymbol u)  \big] \,\mathrm{d}s
\vspace{4pt}\\
\quad\;  =  -\displaystyle\int_{\mathcal{C}_0}  \Big[ e_{\alpha\beta}x_\alpha^\prime x_\beta\big(  P_{ss,s} +  \frac{1}{R} S_s \big) + x_\alpha x_\alpha^\prime \big( S_{s,s} - \frac{1}{R}P_{ss} \big) + \big(S_s+M_{ss,s}\big)  \Big] \mathrm{d}s
\vspace{4pt}\\
\quad\;  =  -  \displaystyle\int_{\mathcal{C}_0} \frac{\partial }{\partial s} \big(e_{\alpha\beta}x_\alpha^\prime x_\beta  P_{ss}   + x_\alpha x_\alpha^\prime  S_{s}   +  M_{ss}\big) \,\mathrm{d}s =   0
\end{array}
$$
	and
$$
\begin{array}{l}
\displaystyle
{\mathcal{M}}_\alpha\left(\frac{\partial \boldsymbol u}{\partial z}\right)  =  - \displaystyle\int_{\mathcal{C}_0} \frac{\partial }{\partial z} \big[e_{\alpha\beta}x_\beta P_{zz}(\boldsymbol u) +x_\alpha^\prime M_{zz}(\boldsymbol u) \big] \,\mathrm{d}s
\vspace{4pt}\\
\qquad =  \displaystyle\int_{\mathcal{C}_0}  \big[e_{\alpha\beta}x_\beta P_{zs,s} +x_\alpha^\prime \big(S_z+ M_{sz,s}\big) \big] \,\mathrm{d}s
\vspace{4pt}\\
\qquad =    -\displaystyle\int_{\mathcal{C}_0}  \Big[ e_{\alpha\beta} \big(x_\beta^\prime P_{sz}  + e_{\beta\gamma}x_\gamma^\prime S_z \big) + \big( \frac{e_{\alpha\beta} x_\beta^\prime}{R}+x_\alpha^{\prime\prime} \big)
M_{sz} \Big] \mathrm{d}s = e_{\alpha\beta}\mathcal{R}_\beta(\boldsymbol u),
\end{array}
$$
since $ x_\alpha^{\prime\prime}= e_{\beta\alpha}x_\beta^{\prime}R^{-1} $.
	\end{proof}
\medskip

By virtue of the linearity of the theory, we can decompose the relaxed Saint-Venant's problem into two problems: the extension-bending-torsion problem (characterized by the resultant loads $ \boldsymbol{\mathcal{R}}^0 = \mathcal{R}^0_3\boldsymbol e_3 $  and $ \boldsymbol{\mathcal{M}}^0 = \mathcal{M}^0_i\boldsymbol e_i $) and the flexure problem (characterized by the resultants of the form $ \boldsymbol{\mathcal{R}}^0 = \mathcal{R}^0_\alpha\boldsymbol e_\alpha $  and $ \boldsymbol{\mathcal{M}}^0 = \boldsymbol 0 $).

\subsection{Extension-bending-torsion problem}
\label{Sect4.1}

The next theorem gives the exact closed-form solution.

\begin{theorem}\label{Th2}
Consider the extension-bending-torsion problem for cylindrical shells: solve the field equations (\ref{eq13})--(\ref{eq16}) under the end edge conditions 
	\begin{equation}\label{eq24}
\boldsymbol{\mathcal{R}}(\boldsymbol u)= \mathcal{R}_3^0\,\boldsymbol{e}_3\,,\qquad \boldsymbol{\mathcal{M}}(\boldsymbol u) = \mathcal{M}_i^0\,\boldsymbol{e}_i\,,
	\end{equation}
	for given resultants $ \mathcal{R}^0_3\, $, $ \mathcal{M}^0_i $ ($ i=1,2,3 $). This problem admits the following solution
	\begin{equation}\label{eq25}
	\begin{array}{l}
	\boldsymbol u = - \big[\,\frac12 A_\alpha(z^2 - \nu x_\beta x_\beta) + \nu(A_\beta x_\beta+ \bar{A}_3)x_\alpha + K z\, e_{\alpha\beta}x_\beta \big]\boldsymbol e_\alpha 
	\vspace{6pt}\\
	\qquad
	+ \big[z(A_\alpha x_\alpha+ \bar{A}_3)  + K \varphi(s) \big]\boldsymbol e_3
	 +\displaystyle\frac{1}{C}\Big(
	s \boldsymbol B\times\boldsymbol e_3 - \int_{0}^{s} \frac{1}{R}(\boldsymbol B\cdot \hat{\boldsymbol r})\boldsymbol{\tau}\,\mathrm{d}s 
	\Big)
	\vspace{6pt}\\
	\displaystyle
	\qquad - \int_{0}^{s} \boldsymbol n \int_{0}^{s} \big[\frac{\nu}{R}\boldsymbol A + \Big( \frac{1}{D}+ \frac{1}{CR^2}\Big)\boldsymbol B  \big]\cdot\hat{\boldsymbol r} \,\mathrm{d}s\,\mathrm{d}s,
	\end{array}
	\end{equation}
	where 
	\begin{equation}\label{eq26}
\hat{\boldsymbol r}= x_\alpha\boldsymbol e_\alpha + \bar{s}\,\boldsymbol e_3\,,\qquad \boldsymbol A=A_i\boldsymbol e_i\,,\qquad \boldsymbol B=B_i\boldsymbol e_i\,,\qquad \bar{A}_3= A_3\bar{s}
	\end{equation}
	and $ A_i\, $, $ B_i $ ($ i=1,2,3 $) and $ K $ are constants. The constants $ A_\alpha $ and $ \bar{A}_3 $ represent measures of curvature and stretch, respectively, which are given in terms of the resultants $ \mathcal{M}^0_\alpha $ and $ \mathcal{R}^0_3 $ by the relations (\ref{eq44-45}). Further, the coefficients $ B_i $ are determined in terms of $ A_i $ through the system of linear algebraic equations (\ref{eq41}). The constant $ K $ is a measure of twist of the cylindrical shell, which is given by the twist-torque relation
	\begin{equation}\label{eq27}
	K=-\frac{\mathcal{M}^0_3}{2(1-\nu)}\left( C\,\frac{\mathcal{A}^2}{\bar{s}} + D\bar{s}\right)^{-1},
	\end{equation}
	while $ \varphi(s) $ is a torsion function having the expression 
	\begin{equation}\label{eq28}
	\varphi(s) = \frac{2\mathcal{A}}{\bar{s}}\,s - \int_0^{s} \boldsymbol r\cdot\boldsymbol n\, \mathrm{d}s,
	\end{equation}
	where $ \displaystyle\mathcal{A} = \frac12 \int_0^{\bar{s}} \boldsymbol r\cdot\boldsymbol n\, \mathrm{d}s $ is the area bounded by the closed cross-section curve $ \mathcal{C}_0 \,$.
\end{theorem}

\begin{proof}
We search for a solution $ \boldsymbol u $  satisfying the equations (\ref{eq13})--(\ref{eq16}) and the conditions (\ref{eq24}), such that $ \frac{\partial \boldsymbol u}{\partial z} \in C^1(\bar{\mathcal S}) \cap C^2({\mathcal S}) $. According to Theorem \ref{Th1}, the field $ \frac{\partial\boldsymbol u  }{\partial z} \,$ verifies then the equilibrium equations (\ref{eq13})--(\ref{eq16})  and the conditions (\ref{eq23}) in the form
$$
\boldsymbol{\mathcal{R}}\left(\frac{\partial \boldsymbol u}{\partial z}\right)=\boldsymbol{0},\qquad \boldsymbol{\mathcal{M}}\left(\frac{\partial \boldsymbol u}{\partial z}\right) = \boldsymbol{0}.
$$
These relations suggest to search for a solution $ \boldsymbol u $ such that the derivative  $ \frac{\partial\boldsymbol u  }{\partial z} \,$ is a rigid body displacement field, i.e.
\begin{equation}\label{eq28,5}
\frac{\partial \boldsymbol u}{\partial z} = \tilde{\boldsymbol c} + \boldsymbol d \times \boldsymbol r\,,
\end{equation}
where $ \tilde{\boldsymbol c}=\tilde{c}_i\boldsymbol e_i $ and $ \boldsymbol d=d_i\boldsymbol e_i $ are constant vectors. It follows
$$
\boldsymbol u = z\,\tilde{\boldsymbol c} +  \boldsymbol d \times \big(x_\alpha z \,\boldsymbol e_\alpha +\frac12\,z^2\boldsymbol e_3\big)+ \tilde{\boldsymbol w}(s),
$$
where $ \tilde{\boldsymbol w}(s) $ is an arbitrary function of class $ C^2[0,\bar{s}] $. If we introduce the notations $ A_\alpha=e_{\beta\alpha}d_\beta\, $, $ \bar{A}_3=A_3\bar{s}=\tilde{c}_3\, $, $ K=d_3\, $, and $ \boldsymbol w(s)=\tilde{\boldsymbol w}(s) + (\tilde{c}_\alpha x_\alpha)\boldsymbol e_3 \,$, then we can writte the last relation in the form
\begin{equation}\label{eq29}
\boldsymbol u = -\big(\frac12\,z^2 A_\alpha  + Kz\,e_{\alpha\beta}x_\beta \big) \boldsymbol e_\alpha + z(\boldsymbol A\cdot\hat{\boldsymbol r})\,\boldsymbol e_3 + \boldsymbol w(s),
\end{equation}
where $ \boldsymbol A = A_i\boldsymbol e_i$ and $ \hat{\boldsymbol r}= x_\alpha\boldsymbol e_\alpha + \bar{s}\,\boldsymbol e_3\, $. In what follows, we determine the unknown function $ \boldsymbol w(s)=w_i(s)\boldsymbol e_i $ such that the displacement field  (\ref{eq29}) satisfies the equilibrium equations (\ref{eq14}). The strain measures (\ref{eq13}) corresponding to the displacement field  (\ref{eq29}) are
\begin{equation}\label{eq30}
\begin{array}{c}
\epsilon_{ss}= \boldsymbol w^\prime\cdot \boldsymbol\tau\,, \qquad 
\epsilon_{sz}= \epsilon_{zs}= \frac12 (K\boldsymbol r\cdot \boldsymbol n + w^\prime_3),\qquad 
\epsilon_{zz}=  \boldsymbol A\cdot\hat{\boldsymbol r}\,,
\vspace{4pt}\\
\rho_{ss}= \boldsymbol w^{\prime\prime}\cdot \boldsymbol n\,, \qquad 
\rho_{sz}= \rho_{zs}= -K,\qquad 
\rho_{zz}=  -\boldsymbol A\cdot \boldsymbol n\,.
\end{array}
\end{equation}
where we denote by $ (\cdot)^\prime \equiv \frac{\mathrm{d}}{\mathrm{d}s}(\cdot) $ the derivative with respect to $ s. $ The stress tensors (\ref{eq16}) have the form
\begin{equation}\label{eq31}
\begin{array}{ll}
 N_{ss}  = C (\nu\boldsymbol A\cdot\hat{\boldsymbol r}+ \boldsymbol w^\prime\cdot \boldsymbol\tau) ,\qquad &
N_{sz} = N_{zs} =\frac12 C(1-\nu) (K\boldsymbol r\cdot \boldsymbol n + w^\prime_3) , 
\vspace{4pt}\\
 N_{zz}  = C (\boldsymbol A\cdot\hat{\boldsymbol r}+ \nu\boldsymbol w^\prime\cdot \boldsymbol\tau) ,
\qquad &
M_{ss}  = D(-\nu\boldsymbol A\cdot \boldsymbol n + \boldsymbol w^{\prime\prime}\cdot \boldsymbol n),
\vspace{4pt}\\
 M_{sz} = M_{zs} = - D(1-\nu) K , \qquad &
M_{zz}  = D(-\boldsymbol A\cdot \boldsymbol n + \nu \boldsymbol w^{\prime\prime}\cdot \boldsymbol n)
\end{array}
\end{equation}
and the relations (\ref{eq15}) yield
\begin{equation}\label{eq32}
\begin{array}{l}
P_{ss}  = \nu\boldsymbol A\cdot( C\hat{\boldsymbol r} + D\,\frac{1}{R}\,\boldsymbol n) + C \boldsymbol w^\prime\cdot \boldsymbol\tau - D\,\frac{1}{R}\,
\boldsymbol w^{\prime\prime}\cdot \boldsymbol n,
\vspace{4pt}\\
P_{sz} = K(1-\nu)(\frac12 C \boldsymbol r\cdot \boldsymbol n + D\,\frac{1}{R} ) + \frac12 C(1-\nu) w^\prime_3 \,, 
\vspace{4pt}\\
P_{zs}  = \frac12 C(1-\nu) (K\boldsymbol r\cdot \boldsymbol n + w^\prime_3),
\qquad
P_{zz} = C (\boldsymbol A\cdot\hat{\boldsymbol r}+ \nu\boldsymbol w^\prime\cdot \boldsymbol\tau) , \vspace{4pt}\\
S_{s} =  D(\nu\boldsymbol A\cdot \boldsymbol n  -\boldsymbol w^{\prime\prime}\cdot \boldsymbol n)^\prime,\qquad\qquad
S_{z}  = 0.
\end{array}
\end{equation}
We note that all these tensor components are independent of $ z $.

Using (\ref{eq32}) we deduce that the equilibrium equations (\ref{eq14})$_2 $ reduces to 
$$
\left( K\,\boldsymbol r\cdot\boldsymbol n + w_3^\prime (s) \,\right)^\prime =0.
$$
Integrating this equations and imposing the continuity condition $ w_3(0)=w_3(\bar{s}) $ (see (\ref{eq39})), we obtain the solution
\begin{equation}\label{eq33}
w_3(s) = K\Big( \frac{2\mathcal{A}}{\bar{s}}\,s - \int_0^{s} \boldsymbol r\cdot\boldsymbol n\, \mathrm{d}s  \Big),\qquad\mbox{i.e.}\qquad w_3(s)=K\,\varphi(s),
\end{equation}
in accordance with notation (\ref{eq28}). 
The remaining two equilibrium equations (\ref{eq14})$_{1,3} $ can be put in the vectorial form
$$
\begin{array}{c}
\big( -P_{ss}\boldsymbol n + S_s \boldsymbol \tau\big)^\prime = \boldsymbol 0, \qquad \mbox{i.e.}
\vspace{6pt}\\
 -P_{ss}\boldsymbol n + S_s \boldsymbol \tau = B_\alpha\boldsymbol e_\alpha\,,\quad
  \mbox{or equivalently},
\end{array}
$$
\begin{equation}\label{eq34}
P_{ss} = -(B_\alpha\boldsymbol e_\alpha)\cdot \boldsymbol n
\qquad\mbox{and}\qquad 
S_s = (B_\alpha\boldsymbol e_\alpha)\cdot \boldsymbol\tau\,,
\end{equation}
where $ B_1\, $, $ B_2 $ are some arbitrary constants. Using (\ref{eq32}), from (\ref{eq34}) we obtain
\begin{equation}\label{eq35}
\begin{array}{l}
C(\nu\boldsymbol A\cdot\hat{\boldsymbol{r}} + \boldsymbol w^\prime\cdot \boldsymbol\tau) + D\,\frac{1}{R}(\nu\boldsymbol A\cdot\boldsymbol n - \boldsymbol w^{\prime\prime}\cdot \boldsymbol n) = -(B_\alpha\boldsymbol e_\alpha)\cdot \boldsymbol n\,,
\vspace{6pt}\\
D (\nu\boldsymbol A\cdot\boldsymbol n - \boldsymbol w^{\prime\prime}\cdot \boldsymbol n) =B_\alpha x_\alpha + B_3 \bar{s}\,,
\end{array}
\end{equation}
where $ B_3 $ is an arbitrary constant. Introducing the notation $  \boldsymbol B = B_i\boldsymbol e_i $ we can write the system of equations (\ref{eq35}) in the simpler form
\begin{equation}\label{eq36}
\begin{array}{rcl}
\displaystyle
\boldsymbol w^\prime\cdot \boldsymbol\tau & = & -\nu\boldsymbol A\cdot\hat{\boldsymbol{r}} -\frac{1}{C}\,\boldsymbol B\cdot\big(\boldsymbol n + \frac{1}{R}\,\hat{\boldsymbol{r}}\big)
,
\vspace{6pt}\\
\displaystyle
\boldsymbol w^{\prime\prime}\cdot \boldsymbol n & = & \nu\boldsymbol A\cdot\boldsymbol n - \frac{1}{D}\,\boldsymbol B\cdot\hat{\boldsymbol{r}}.
\end{array}
\end{equation}
Using (\ref{eq36}), we can integrate the relation $ (\boldsymbol w^\prime\cdot\boldsymbol n)^\prime = \boldsymbol w^{\prime\prime}\cdot\boldsymbol n + \boldsymbol w^\prime\cdot\boldsymbol \tau\,\frac{1}{R}\, $ to deduce
\begin{equation}\label{eq37}
\boldsymbol w^\prime\cdot\boldsymbol n = \nu\boldsymbol A\cdot\Big( e_{\alpha\beta}x_\beta\boldsymbol e_\alpha - \int_{0}^{s} \frac{1}{R}\,\hat{\boldsymbol r}\,\mathrm{d}s  \Big) +
\boldsymbol B\cdot\Big[ \frac{1}{C}\,\boldsymbol{\tau} - \int_{0}^{s} \Big( \frac{1}{D}+ \frac{1}{CR^2}\Big)\hat{\boldsymbol r}\,\mathrm{d}s  \Big],
\end{equation}
where we neglect some additive constants which represent rigid body displacement fields. The equations (\ref{eq36})$ _1 $ and (\ref{eq37}) determine the field $ \boldsymbol w^\prime(s) $. If we integrate once again with respect to $ s $, we obtain the following solution for $ \boldsymbol w(s)\, $:
\begin{equation}\label{eq38}
\begin{array}{l}
\displaystyle 
w_\alpha(s)\,\boldsymbol e_\alpha =  \frac12 \nu ( x_\gamma x_\gamma)A_\alpha\boldsymbol e_\alpha - \nu(\boldsymbol A\cdot \hat{\boldsymbol r})x_\alpha \boldsymbol e_\alpha
+ \frac{1}{C}\Big[
s \boldsymbol B\times\boldsymbol e_3 
- \int_{0}^{s} \frac{1}{R}(\boldsymbol B\cdot \hat{\boldsymbol r})\boldsymbol{\tau}\,\mathrm{d}s 
\Big]
\vspace{6pt}\\
\displaystyle \qquad\qquad\quad
 - \int_{0}^{s} \boldsymbol n \int_{0}^{s} \big[\frac{\nu}{R}\boldsymbol A + \Big( \frac{1}{D}+ \frac{1}{CR^2}\Big)\boldsymbol B  \big]\cdot\hat{\boldsymbol r} \,\mathrm{d}s\,\mathrm{d}s.
\end{array}
\end{equation}
Inserting the relations (\ref{eq33}) and (\ref{eq38}) into equation (\ref{eq29}) we derive that the solution $ \boldsymbol u $ has indeed the expression (\ref{eq25}), which was announced in the statement of the theorem. This solution must satisfy some continuity conditions for $ s=0,\bar{s}\, $, namely
\begin{equation}\label{eq39}
\boldsymbol u(0,z)= \boldsymbol u(\bar{s},z),\qquad  \boldsymbol u_{,s}(0,z)= \boldsymbol u_{,s}(\bar{s},z),\qquad  \boldsymbol u_{,ss}(0,z)= \boldsymbol u_{,ss}(\bar{s},z),
\end{equation}
for all $ z\in [0,\ell] $. In view of the integrations that we performed to obtain relations (\ref{eq37}) and (\ref{eq38}), the conditions (\ref{eq39}) reduce to
\begin{equation}\label{eq40}
\boldsymbol w^\prime\cdot\boldsymbol n(0)= \boldsymbol w^\prime\cdot\boldsymbol n(\bar{s}) 
\qquad\mbox{and}\qquad 
w_\alpha(0)=w_\alpha(\bar{s}),\quad \alpha=1,2.
\end{equation}
By virtue of (\ref{eq37}) and (\ref{eq38}), the  last relations (\ref{eq40}) can be written in the form
\begin{equation}\label{eq41}
\begin{array}{l}
\displaystyle\nu\boldsymbol A\cdot\int_{0}^{\bar{s}} \frac{1}{R}\,\hat{\boldsymbol r}\,\mathrm{d}s+ 
\boldsymbol B\cdot\int_{0}^{\bar{s}} \Big(\frac{1}{D}+\frac{1}{CR^2}\Big)\hat{\boldsymbol r}\,\mathrm{d}s = 0,
\vspace{6pt}\\
\displaystyle\nu\boldsymbol A\cdot\int_{0}^{\bar{s}} \frac{x_\alpha}{R}\,\hat{\boldsymbol r}\,\mathrm{d}s+ 
\boldsymbol B\cdot\int_{0}^{\bar{s}} \Big[x_\alpha\Big(\frac{1}{D}+\frac{1}{CR^2}\Big)\hat{\boldsymbol r} + \frac{1}{C} \big( \boldsymbol e_\alpha+ x_\alpha^\prime \boldsymbol \tau\big)\Big]\,\mathrm{d}s = 0.
\end{array}
\end{equation}
From the system of three linear algebaric equations (\ref{eq41}) one can determine the constants $ B_i $ in terms of $ A_i $ ($ i=1,2,3 $).

To conclude the proof, we determine the constants $ A_i $ and $ K $ in terms of the resultant loads $ \mathcal{R}_3^0\,$ and $ \mathcal{M}_i^0\, $. We remark first that the functionals (\ref{eq21}) can be written in another form: namely, inserting (\ref{eq15}) into (\ref{eq21}) and making some transformations we obtain
\begin{equation}\label{eq42}
\begin{array}{l}
\boldsymbol{\mathcal{R}}(\boldsymbol u)= -\displaystyle\int_{\mathcal{C}_0} \big[N_{sz} \boldsymbol \tau + N_{zz} \boldsymbol e_3 - M_{zz,z} \boldsymbol n  \big]  \mathrm{d}s \qquad\quad \mbox{and}
\vspace{6pt}\\ 
\boldsymbol{\mathcal{M}}(\boldsymbol u) = -\!\!\displaystyle\int_{\mathcal{C}_0} \!\big[ N_{zz} \boldsymbol r\times \boldsymbol e_3 + M_{zz}\boldsymbol \tau + \big( N_{sz}\boldsymbol r\cdot \boldsymbol n -2 M_{sz} + M_{zz,z} \boldsymbol r\cdot \boldsymbol\tau \big)\boldsymbol e_3
\big] \mathrm{d}s.
\end{array}
\end{equation}
On the other hand, using (\ref{eq33}) and (\ref{eq36}), we obtain the following expressions for the stress tensors (\ref{eq31}) corresponding to the solution $ \boldsymbol u\, $:
\begin{equation}\label{eq43}
\begin{array}{ll}
\displaystyle
N_{ss}  = -\boldsymbol B\cdot\big(\boldsymbol n + \frac{1}{R}\,\hat{\boldsymbol{r}}\big),\qquad & \displaystyle
N_{sz} = N_{zs} =  C(1-\nu) K\, \frac{\mathcal{A}}{\bar{s}}, 
\vspace{4pt}\\
\displaystyle
N_{zz}  = C (1-\nu^2)\boldsymbol A\cdot\hat{\boldsymbol{r}} -\nu\boldsymbol B\cdot \big(\boldsymbol n + \frac{1}{R}\,\hat{\boldsymbol{r}}\big),
\qquad & \displaystyle
M_{ss}  = -\boldsymbol B\cdot\hat{\boldsymbol{r}},
\vspace{4pt}\\
\displaystyle
M_{zz}  = -D(1-\nu^2)\boldsymbol A\cdot\boldsymbol n - \nu\boldsymbol B\cdot\hat{\boldsymbol{r}}
 , \qquad & \displaystyle
M_{sz} = M_{zs} = - D(1-\nu) K.
\end{array}
\end{equation}
Inserting (\ref{eq43}) in (\ref{eq42}), we derive after some calculations that the end edge conditions (\ref{eq24}) reduce to the following equations
\begin{equation}\label{eq44-45}
\begin{array}{l}
\displaystyle C(1-\nu^2)\boldsymbol A\cdot\int_{0}^{\bar{s}} \hat{\boldsymbol r}\,\mathrm{d}s - \nu
\boldsymbol B\cdot\int_{0}^{\bar{s}} \frac{1}{R}\,\hat{\boldsymbol r}\,\mathrm{d}s = -\mathcal{R}_3^0\,,
\vspace{6pt}\\
\displaystyle C(1-\nu^2)\boldsymbol A\cdot\int_{0}^{\bar{s}} \Big(x_\alpha\,\hat{\boldsymbol r} + \frac{D}{C}\,e_{\alpha\beta}x_\beta^\prime \boldsymbol n\Big)\mathrm{d}s 
\vspace{6pt}\\
\qquad\qquad\qquad\displaystyle 
 - \nu
\boldsymbol B\cdot\int_{0}^{\bar{s}} \Big[ x_\alpha \big(\frac{1}{R}\,\hat{\boldsymbol{r}} + \boldsymbol n\big) - e_{\alpha\beta}x_\beta^\prime \hat{\boldsymbol r}\Big]\,\mathrm{d}s = e_{\alpha\beta}\mathcal{M}_\beta^0\,,
\end{array}
\end{equation}
together with the twist-torque relation (\ref{eq27}). The equations (\ref{eq44-45}) represent a linear algebraic system for the determination of the constants $ A_i $ ($ i=1,2,3 $) in terms of $ \mathcal{R}_3^0\,, \mathcal{M}_1^0\,, \mathcal{M}_2^0 $ (since $ \boldsymbol B $ has already been determined in (\ref{eq41})). The proof is complete.
\end{proof}
\medskip

In Section \ref{Sect5} we present an approximation of the solution $ \boldsymbol u $ given by Theorem \ref{Th2} (in the thin shell limit $ h\rightarrow 0 $) and show the interpretation of the constants $ A_i $ as measures of stretch and curvature of the cylindrical shell considered as a beam.

\subsection{Flexure problem}
\label{Sect4.2}

The explicit solution in closed-form is stated in the next result.

\begin{theorem}\label{Th3}
	The flexure problem for cylindrical shells, consisting in the governing equations (\ref{eq13})--(\ref{eq16}) and the end edge conditions 
	\begin{equation}\label{eq46}
	\boldsymbol{\mathcal{R}}(\boldsymbol u)= \mathcal{R}_\alpha^0\,\boldsymbol{e}_\alpha\,,\qquad \boldsymbol{\mathcal{M}}(\boldsymbol u) = \boldsymbol{0},
	\end{equation}
	admits the following solution
	\begin{equation}\label{eq47}
\begin{array}{l}
\hat{\boldsymbol u} = - \big[(\frac16 z^3 - \frac12\nu z x_\beta x_\beta)\hat A_\alpha + \nu z(\hat{\boldsymbol A}\cdot \hat{\boldsymbol r})x_\alpha + \tilde{K} z\, e_{\alpha\beta}x_\beta \big]\boldsymbol e_\alpha
\vspace{6pt}\\
\quad\,  + \big[\frac12 z^2(\hat{\boldsymbol A}\cdot \hat{\boldsymbol r})  + \tilde{K} \varphi(s) + \psi(s)\big]\boldsymbol e_3
+\displaystyle\frac{1}{C}\,z\Big(
s \hat{\boldsymbol B}\times\boldsymbol e_3 - \int_{0}^{s} \frac{1}{R}(\hat{\boldsymbol B}\cdot \hat{\boldsymbol r})\boldsymbol{\tau}\,\mathrm{d}s 
\Big) 
\vspace{6pt}\\
\quad\;\displaystyle
 - \,z\int_{0}^{s} \boldsymbol n \int_{0}^{s} \big[\frac{\nu}{R}\hat{\boldsymbol A} + \Big( \frac{1}{D}+ \frac{1}{CR^2}\Big)\hat{\boldsymbol B}  \big]\cdot\hat{\boldsymbol r} \,\mathrm{d}s\,\mathrm{d}s,
\end{array}
	\end{equation}
	where  $\hat{\boldsymbol A}= \hat{A}_i\boldsymbol e_i\, $, $ \hat{\boldsymbol B}= \hat{B}_i\boldsymbol e_i $ are constant vectors and the constant $ \tilde{K} $ is given by relation (\ref{eq57}). The constant coefficients $ \hat{B}_i $ are determined in terms of $ \hat{A}_i $ by the system of equations (\ref{eq41}), while the constants $ \hat{A}_i $ can be calculated in terms of the resultant forces $ \mathcal{R}_\alpha^0 $ from the equations (\ref{eq49}). In the above solution, $ \varphi(s) $ is the torsion function (\ref{eq28}) and $ \psi(s) $ is a flexure function  given by (\ref{eq54}).
\end{theorem}

\begin{proof}
	To find a solution $\hat{\boldsymbol u}$  of the flexure problem, we employ the results of Theorems \ref{Th1} and \ref{Th2}. We note that the solution $ \boldsymbol u $ given in Theorem \ref{Th2} by relation (\ref{eq25}) depend on four constants $ A_i $ and $ K $. To show this dependence explicitly, we denote the displacement field (\ref{eq25}) by $ \boldsymbol u[\boldsymbol A,K] $. 
	
	On the other hand, from Theorem \ref{Th1} we deduce that $ \frac{\partial\hat{\boldsymbol u}  }{\partial z} \,$ satisfies the governing equations (\ref{eq13})--(\ref{eq16}) together with the end edge conditions
	$$
\boldsymbol{\mathcal{R}}\left(\frac{\partial \hat{\boldsymbol u}}{\partial z}\right)=\boldsymbol{0},\qquad \boldsymbol{\mathcal{M}}\left(\frac{\partial \hat{\boldsymbol u}}{\partial z}\right) = e_{\alpha\beta}\mathcal{R}_\beta^0\,\boldsymbol{e}_\alpha\,.
	$$
	Thus, $ \frac{\partial\hat{\boldsymbol u}  }{\partial z} \,$ is a solution of the bending problem with resultant bending moment $ \boldsymbol{\mathcal{M}}^0 = e_{\alpha\beta}\mathcal{R}_\beta^0\,\boldsymbol{e}_\alpha\, $.
Here, the resultant force and twisting moment are vanishing, i.e. $ \boldsymbol{\mathcal{R}}^0=\boldsymbol{0} $ and 
$ \mathcal{M}_3^0 =0$.
	
	Combining the results of Theorems \ref{Th1} and \ref{Th2}, we can choose 
		\begin{equation}\label{eq48}
		\frac{\partial\hat{\boldsymbol u}  }{\partial z} \,=\,\boldsymbol u[\hat{\boldsymbol A},\hat{K}]\,,
		\end{equation}
	where the constants $ \hat{A}_i $  are given by  (in view of (\ref{eq44-45}) written for $ \mathcal{R}_3^0=0 $ and bending moments $ e_{\alpha\beta}\mathcal{R}_\beta^0 \,$)
	\begin{equation}\label{eq49}
\begin{array}{l}
\displaystyle C(1-\nu^2)\hat{\boldsymbol A}\cdot\int_{0}^{\bar{s}} \hat{\boldsymbol r}\,\mathrm{d}s - \nu
\hat{\boldsymbol B}\cdot\int_{0}^{\bar{s}} \frac{1}{R}\,\hat{\boldsymbol r}\,\mathrm{d}s = 0\,,
\vspace{6pt}\\
\displaystyle C(1-\nu^2)\hat{\boldsymbol A}\cdot\int_{0}^{\bar{s}} \Big(x_\alpha\,\hat{\boldsymbol r} + \frac{D}{C}\,e_{\alpha\beta}x_\beta^\prime \boldsymbol n\Big)\mathrm{d}s 
\vspace{6pt}\\
\qquad\qquad\qquad\displaystyle 
- \nu
\hat{\boldsymbol B}\cdot\int_{0}^{\bar{s}} \Big[ x_\alpha \big(\frac{1}{R}\,\hat{\boldsymbol{r}} + \boldsymbol n\big) - e_{\alpha\beta}x_\beta^\prime \hat{\boldsymbol r}\Big]\,\mathrm{d}s = -\mathcal{R}_\alpha^0\,,
\end{array}
	\end{equation}
	together with relations (\ref{eq41}) (written with $ \hat{\boldsymbol A} , \hat{\boldsymbol B} $ instead of $  {\boldsymbol A} ,  {\boldsymbol B} $). By virtue of (\ref{eq27}) (written with twisting moment $ \mathcal{M}_3^0=0 $), we find 
\begin{equation}\label{eq49,5}
\hat{K}=0.
\end{equation}
	
	Taking into account that $ \frac{\partial{\boldsymbol u}[\boldsymbol A,K]  }{\partial z} \, $ is a rigid body displacement field (see (\ref{eq28,5})), we obtain from the relation (\ref{eq48}) and (\ref{eq49,5}) by integration with respect to $ z $ :
	\begin{equation}\label{eq50}
\hat{\boldsymbol u} = \int_{0}^{z}	\boldsymbol u[\hat{\boldsymbol A},0]\,\mathrm{d}z + \boldsymbol u[\tilde{\boldsymbol A},\tilde{K}] + \hat{\boldsymbol w}(s),
	\end{equation}
	where $ \tilde{\boldsymbol A}=\tilde{A}_i \boldsymbol e_i $ and $ \tilde{K} $ are constants, while $ \hat{\boldsymbol w}(s)= \hat{w}_i(s)\boldsymbol e_i $ is an arbitrary function of class $ C^2[0,\bar{s}] $.
	Suggested by the previous results concerning the flexure problem in three-dimensional elasticity \cite{Iesan86,Iesan87} and Cosserat shell theory \cite{Birsan-JE-04}, we search for the solution $ \hat{\boldsymbol u}  $ such that
	\begin{equation}\label{eq51}
\tilde{\boldsymbol A}= \boldsymbol 0\qquad \mbox{and}\qquad \hat{w}_\alpha(s)=0.
	\end{equation}
	The choice (\ref{eq51}) is made for the sake of simplity. (Alternatively, we can work with the unknown vector $ \tilde{\boldsymbol A} $ and functions $ \hat{w}_\alpha(s) $ and obtain finally from the equilibrium equations and end edge conditions that the relations (\ref{eq51}) hold.)
	Thus, using (\ref{eq51}) in (\ref{eq50}), we search for a solution $ \hat{\boldsymbol u}  $  of the flexure problem in the form
	\begin{equation}\label{eq52}
\hat{\boldsymbol u} = \int_{0}^{z}	\boldsymbol u[\hat{\boldsymbol A},0]\,\mathrm{d}z + \boldsymbol u[\boldsymbol 0,\tilde{K}] + \psi(s)\boldsymbol e_3\,,
	\end{equation}
	where we have denoted $ \psi(s)=\hat{w}_3(s) $ for brevity.
	
	In what follows, we shall determine the unknown function $ \psi(s) $
	from the equilibrium equations (\ref{eq14}) and the constant $ \tilde{K} $ from the end edge conditions (\ref{eq46}). Inserting the relation (\ref{eq25}) into (\ref{eq52}), we deduce from the geometrical equations (\ref{eq13}) the following strain measures corresponding to $ \hat{\boldsymbol u} \,$:
	\begin{equation}\label{eq53}
\begin{array}{l}
\displaystyle
\hat\epsilon_{ss}= -z\big[  \nu\hat{\boldsymbol A}\cdot\hat{\boldsymbol{r}} +\frac{1}{C}\,\hat{\boldsymbol B}\cdot\big(\boldsymbol n + \frac{1}{R}\,\hat{\boldsymbol{r}}\big)  \big], \qquad 
\hat\epsilon_{zz}=  z(\hat{\boldsymbol A}\cdot\hat{\boldsymbol r}),
\vspace{4pt}\\
\displaystyle
\hat\epsilon_{sz}= \hat\epsilon_{zs}= \frac12 \Big[ \frac{\nu}{2}(x_\alpha x_\alpha)\hat{\boldsymbol A} -\nu(\hat{\boldsymbol A}\cdot\hat{\boldsymbol{r}})\boldsymbol r    +  \frac{1}{C}\Big(s\hat{\boldsymbol B}\times\boldsymbol e_3 -   \int_{0}^{s} \frac{1}{R}(\hat{\boldsymbol B}\cdot \hat{\boldsymbol r})\boldsymbol{\tau}\,\mathrm{d}s\Big) 
\vspace{4pt}\\
\displaystyle
\qquad\quad -\int_{0}^{s} \boldsymbol n \int_{0}^{s} \big[\frac{\nu}{R}\hat{\boldsymbol A} + \Big( \frac{1}{D}+ \frac{1}{CR^2}\Big)\hat{\boldsymbol B}  \big]\cdot\hat{\boldsymbol r} \,\mathrm{d}s\,\mathrm{d}s                         
\Big]\cdot \boldsymbol{\tau}
+\tilde{K}\,\frac{\mathcal{A}}{\bar{s}}+ \frac12\,\psi^\prime (s),
\vspace{4pt}\\
\displaystyle
\hat\rho_{ss}= z\big(  \nu\hat{\boldsymbol{A}}
\cdot\boldsymbol n - \frac{1}{D}\,\hat{\boldsymbol B}\cdot\hat{\boldsymbol{r}} \big), \qquad 
\hat\rho_{zz}=  - z (\hat{\boldsymbol A}\cdot \boldsymbol n),
\vspace{4pt}\\
\displaystyle
\hat\rho_{sz}= \hat\rho_{zs}= - \tilde{K} + \nu\hat{\boldsymbol A}\cdot(e_{\alpha\beta}x_\beta\boldsymbol e_\alpha) + 
\frac{1}{C}   \hat{\boldsymbol B}\cdot  \boldsymbol{\tau} -\!\! \int_{0}^{s}\! \big[\frac{\nu}{R}\hat{\boldsymbol A} + \Big( \frac{1}{D}+ \frac{1}{CR^2}\Big)\hat{\boldsymbol B}  \big]\!\cdot\hat{\boldsymbol r} \,\mathrm{d}s.
\end{array}
	\end{equation}
	We note that the dependence of these tensors on $ z $ is at most linear. Using the expressions (\ref{eq53}) we can calculate the stress tensors according to relations (\ref{eq15}) and (\ref{eq16}), and then write the equilibrium equations (\ref{eq14}). We obtain that the equilibrium equations (\ref{eq14})$ _{1,3} $ are identically satisfied, while the equilibrium equation (\ref{eq14})$ _{2} $ reduces to
	$$
C(1-\nu) \big(\hat\epsilon_{sz}\big)^\prime + C (1-\nu^2)\hat{\boldsymbol A}\cdot\hat{\boldsymbol{r}} -\nu\hat{\boldsymbol B}\cdot \big(\boldsymbol n + \frac{1}{R}\,\hat{\boldsymbol{r}}\big) =0,
	$$
where $ \hat\epsilon_{sz} $	is given by (\ref{eq53})$_3\, $.
	Integrating the last equation two times with respect to $ s $, we find the solution
	\begin{equation}\label{eq54}
\begin{array}{l}
\displaystyle
\psi(s) = -\int_{0}^{s}  \Big\{ \boldsymbol{\tau}\cdot \Big[
 \frac{\nu}{2}(x_\alpha x_\alpha)\hat{\boldsymbol A} -\nu(\hat{\boldsymbol A}\cdot\hat{\boldsymbol{r}})\boldsymbol r    +  \frac{1}{C}\Big(s\hat{\boldsymbol B}\times\boldsymbol e_3 -   \int_{0}^{s} \frac{1}{R}(\hat{\boldsymbol B}\cdot \hat{\boldsymbol r})\boldsymbol{\tau}\,\mathrm{d}s\Big)
 \vspace{4pt}\\
 \displaystyle
 \qquad\qquad
 -\int_{0}^{s} \boldsymbol n \int_{0}^{s} \big[\frac{\nu}{R}\hat{\boldsymbol A} + \Big( \frac{1}{D}+ \frac{1}{CR^2}\Big)\hat{\boldsymbol B}  \big]\cdot\hat{\boldsymbol r} \,\mathrm{d}s\,\mathrm{d}s   \Big]
\vspace{4pt}\\
\displaystyle
\qquad\qquad +\frac{2}{C(1-\nu)}  \int_{0}^{s} \big[C (1-\nu^2)\hat{\boldsymbol A}\cdot\hat{\boldsymbol{r}} -\nu\hat{\boldsymbol B}\cdot \big(\boldsymbol n + \frac{1}{R}\,\hat{\boldsymbol{r}}\big) \big]\mathrm{d}s                    
\Big\}\mathrm{d}s 
+
{K}_0s,
\end{array}
	\end{equation}
	where the constant $ K_0 $ is determined by the continuity condition $ \psi(0)=\psi(\bar{s}) $ as
	\begin{equation}\label{eq55}
	\begin{array}{l}
	\displaystyle
	K_0 = \frac{1}{\bar{s}} \int_{0}^{\bar{s}}  \Big\{ \boldsymbol{\tau}\cdot \Big[
	\frac{\nu}{2}(x_\alpha x_\alpha)\hat{\boldsymbol A} -\nu(\hat{\boldsymbol A}\cdot\hat{\boldsymbol{r}})\boldsymbol r    +  \frac{1}{C}\Big(s\hat{\boldsymbol B}\times\boldsymbol e_3 -   \int_{0}^{s} \frac{1}{R}(\hat{\boldsymbol B}\cdot \hat{\boldsymbol r})\boldsymbol{\tau}\,\mathrm{d}s\Big)
	\vspace{4pt}\\
	\displaystyle
	\qquad\qquad
	-\int_{0}^{s} \boldsymbol n \int_{0}^{s} \big[\frac{\nu}{R}\hat{\boldsymbol A} + \Big( \frac{1}{D}+ \frac{1}{CR^2}\Big)\hat{\boldsymbol B}  \big]\cdot\hat{\boldsymbol r} \,\mathrm{d}s\,\mathrm{d}s   \Big]
	\vspace{4pt}\\
	\displaystyle
	\qquad\qquad +\frac{2}{C(1-\nu)}  \int_{0}^{s} \big[C (1-\nu^2)\hat{\boldsymbol A}\cdot\hat{\boldsymbol{r}} -\nu\hat{\boldsymbol B}\cdot \big(\boldsymbol n + \frac{1}{R}\,\hat{\boldsymbol{r}}\big) \big]\mathrm{d}s                    
	\Big\}\mathrm{d}s .
	\end{array}
	\end{equation}
	We mention that the conditions $ \psi^\prime(0)=\psi^\prime(\bar{s}) $  and $ \psi^{\prime\prime}(0)=\psi^{\prime\prime}(\bar{s}) $ are also satisfied, by virtue of the relation (\ref{eq49})$ _{1}\, $. 
	
	Finally,  we impose the end edge conditions (\ref{eq46}) in order to determine the constant $ \tilde{K} $.
	In view of (\ref{eq53}), (\ref{eq54}) and (\ref{eq15}), the stress tensors $ \hat{\boldsymbol N} $ and $ \hat{\boldsymbol M} $ corresponding to the solution $ \hat{\boldsymbol u} $ have the components
	\begin{equation}\label{eq56}
\begin{array}{l}
\displaystyle
\hat N_{ss}  = -z\, \hat{\boldsymbol B}\cdot\big(\boldsymbol n + \frac{1}{R}\,\hat{\boldsymbol{r}}\big),\qquad  
\hat N_{zz}  = z \big[C (1-\nu^2)\hat{\boldsymbol A}\cdot\hat{\boldsymbol{r}} -\nu\hat{\boldsymbol B}\cdot \big(\boldsymbol n + \frac{1}{R}\,\hat{\boldsymbol{r}}\big) \big],
\vspace{4pt}\\
\displaystyle
\hat N_{sz} = \hat N_{zs} =  C(1-\nu) \Big(\tilde{K}\frac{\mathcal{A}}{\bar{s}} + \frac12 K_0 \Big) - \!\int_{0}^{s}\!\! \big[C (1\!-\!\nu^2)\hat{\boldsymbol A}\cdot\hat{\boldsymbol{r}} -\nu\hat{\boldsymbol B}\cdot \!\big(\boldsymbol n \!+ \frac{1}{R}\,\hat{\boldsymbol{r}}\big) \big]\mathrm{d}s , 
\vspace{4pt}\\
\displaystyle
\hat M_{ss}  = -z(\hat{\boldsymbol B}\cdot\hat{\boldsymbol{r}}),\qquad
\hat M_{zz}  = - z \big[D(1-\nu^2)\hat{\boldsymbol A}\cdot\boldsymbol n + \nu\hat{\boldsymbol B}\cdot\hat{\boldsymbol{r}}\big]
,
\vspace{4pt}\\
 \displaystyle
\hat M_{sz} = \hat M_{zs} =  D(1-\nu) \Big(- \tilde{K} + \nu\hat{\boldsymbol A}\cdot(\boldsymbol r\times\boldsymbol e_3) + 
\frac{1}{C}   \hat{\boldsymbol B}\cdot  \boldsymbol{\tau} 
\vspace{4pt}\\
\displaystyle
\qquad\qquad\qquad\qquad\qquad\qquad\qquad
- \int_{0}^{s} \big[\frac{\nu}{R}\hat{\boldsymbol A} + \Big( \frac{1}{D}+ \frac{1}{CR^2}\Big)\hat{\boldsymbol B}  \big]\!\cdot\hat{\boldsymbol r} \,\mathrm{d}s\Big).
\end{array}
	\end{equation}
	Using the expression (\ref{eq42})$ _{1} \,$ and (\ref{eq56}), we remark that the end edge condition $ \boldsymbol{\mathcal{R}} (\hat{\boldsymbol u} ) = \mathcal{R}_\alpha^0 \boldsymbol e_\alpha $ is satisfied, since it reduces to the equations (\ref{eq49})$ _{2}\, $. On the other hand, the remaining end edge condition $ \boldsymbol{\mathcal{M}}(\hat{\boldsymbol u} ) = \boldsymbol 0 $
	yields (with the help of relations (\ref{eq42})$ _{2} $ and (\ref{eq56})) an algebraic equation for the determination of $ \tilde{K} $. Thus, we find the value
	\begin{equation}\label{eq57}
\begin{array}{l}
\displaystyle
\tilde{K}\Big[ 2(1-\nu)\Big( C\frac{\mathcal{A}^2}{\bar{s}}+ D\bar{s}\Big)\Big] = -C(1-\nu)\mathcal{A} K_0  
+ 2D(1-\nu)\int_{0}^{\bar{s}} \Big\{ \nu\hat{\boldsymbol A}\cdot(\boldsymbol r\times\boldsymbol e_3) 
\vspace{4pt}\\
\displaystyle
- \!\int_{0}^{s} \!\big[\frac{\nu}{R}\hat{\boldsymbol A} +\! \Big( \frac{1}{D}+ \frac{1}{CR^2}\Big)\hat{\boldsymbol B}  \big]\!\cdot\hat{\boldsymbol r}\, \mathrm{d}s \Big\}\mathrm{d}s +\!
\int_{0}^{\bar{s}}\!\! \Big\{
(\boldsymbol r\cdot\boldsymbol{\tau}) \big[D(1-\nu^2)\hat{\boldsymbol A}\cdot\boldsymbol n + \nu\hat{\boldsymbol B}\cdot\hat{\boldsymbol{r}}\big]
\vspace{4pt}\\
\displaystyle
\qquad\qquad\qquad\qquad\qquad
+ (\boldsymbol r\cdot\boldsymbol{n}) \int_{0}^{s} \big[C (1-\nu^2)\hat{\boldsymbol A}\cdot\hat{\boldsymbol{r}} -\nu\hat{\boldsymbol B}\cdot \big(\boldsymbol n + \frac{1}{R}\,\hat{\boldsymbol{r}}\big) \big]\mathrm{d}s
\Big\}\mathrm{d}s,
\end{array}
	\end{equation}
	where the constant $ K_0 $ has been determined in (\ref{eq55}). This concludes the proof. 
\end{proof}
\medskip

In the next section we present an approximated form of the exact closed-form solution for Saint-Venant's problem, which could be more useful in applications. The relative errors contained in the simplified solution are smaller than the errors contained unavoidably in the relations of any first-approximation shell theory.

\section{The simplified solution}
\label{Sect5}

Let us choose the origin $ O $ of the Cartesian coordinate system in the centroid of the cross-section, i.e. we assume that
\begin{equation}\label{eq59}
\int_{0}^{\bar{s}} x_\alpha(s)\,\mathrm{d}s =0\quad (\alpha=1,2),\quad\mbox{and let}\qquad I_{\alpha\beta}=\int_{0}^{\bar{s}} x_\alpha x_\beta\,\mathrm{d}s\,.
\end{equation}
We consider that the shell is very thin, i.e. $ \frac{h}{R}\ll 1 $. Thus, we shall neglect some small terms which are of higher order in $ \frac{h}{R} $ and find a simpler form of the solution. For instance, we have
\begin{equation}\label{eq58}
\frac{1}{D}+ \frac{1}{CR^2}= \frac{1}{D}\Big(  1 + \frac{D}{CR^2} \Big) \simeq \frac{1}{D}\;,
\end{equation}
since $ \frac{D}{CR^2} = \frac{1}{12}\big( \frac{h}{R} \big)^2 \ll 1 $. Furthermore, we consider that the length $ \bar{s} $ of the cross-section curve $ \mathcal{C}_0 $ is of the same order of magnitude as $ R $, while the area $ \mathcal{A} $ of the cross-section is of the same order as $ \bar{s}^2 $ (for instance, for a circular tube $ \bar{s}=2\pi R $ and $ \mathcal{A}= \frac14 \bar{s}^2 $). Hence, we have $ \frac{h}{\bar{s}}\ll 1 $ and we can approximate the torsional rigidity in (\ref{eq27}) as follows
$$
C\,\frac{\mathcal{A}^2}{\bar{s}} + D\bar{s} = C\,\frac{\mathcal{A}^2}{\bar{s}} \Big( 1 + \frac{D}{C}\, \frac{\bar{s}^2}{\mathcal{A}^2} \Big) \;\simeq\; C\,\frac{\mathcal{A}^2}{\bar{s}}\,,
$$
since $ \frac{D}{C}\,\frac{\bar{s}^2}{\mathcal{A}^2} = \frac{h^2}{12}\,\frac{\bar{s}^2}{\mathcal{A}^2}= O\big( (\frac{h}{\bar{s}})^2\big)\ll 1 $. Then , the twist-torque relation  (\ref{eq27})  simplifies to
\begin{equation}\label{eq62}
K = -\frac{\mathcal{M}_3^0\,\bar{s}}{2C(1-\nu)\mathcal{A}^2 }
= -\frac{\mathcal{M}_3^0\,\bar{s}}{4\mu h\mathcal{A}^2}\;,
\end{equation}
where $ \mu = \frac{E}{2(1+\nu)} $ is the shear modulus (Lam\'e's constant). Thus, $ K  $ is a global measure of twist for the cylindrical shell.

In view of (\ref{eq59}) and (\ref{eq58}) one can write the system of equations (\ref{eq41}) in the simplified form 
\begin{equation}\label{eq1o}
\begin{array}{rcl}
I_{\alpha\beta}\,B_\beta & = & -\nu D\Big( A_\beta  \displaystyle  \int_{0}^{\bar{s}} \frac{x_\alpha x_\beta}{R}\,\mathrm{d}s  + A_3\bar{s}\int_{0}^{\bar{s}} \frac{x_\alpha}{R}\,\mathrm{d}s \Big)  ,
\vspace{4pt}\\
B_3 & = & \displaystyle  -\frac{\nu D}{\bar{s}^2} \Big( A_\alpha \int_{0}^{\bar{s}} \frac{x_\alpha}{R}\,\mathrm{d}s   + A_3\bar{s}\, 2\pi \Big)  ,
\end{array}
\end{equation}
since $   \int_{0}^{\bar{s}} \frac{1}{R}\,\mathrm{d}s = 2\pi $. From equations (\ref{eq1o}) we determine the constants $ B_i $ in terms of $ A_i $ ($ i=1,2,3 $) and we observe that $ \boldsymbol B $ has the same order of magnitude as $ \frac{D}{\bar{s}}\,\boldsymbol A\, $. This remark allows us to neglect some small terms in the equations (\ref{eq44-45}) and to write the system (\ref{eq44-45}) in the approximated form
$$
\displaystyle C(1-\nu^2)\boldsymbol A\cdot\int_{0}^{\bar{s}} \hat{\boldsymbol r}\,\mathrm{d}s  = -\mathcal{R}_3^0\,, \qquad
\displaystyle C(1-\nu^2)\boldsymbol A\cdot\int_{0}^{\bar{s}} x_\alpha\,\hat{\boldsymbol r}\,\mathrm{d}s = e_{\alpha\beta}\mathcal{M}_\beta^0\,,
$$
or equivalently, in view of (\ref{eq59}) and $ C(1-\nu^2)=Eh $,
\begin{equation}\label{eq61}
I_{\alpha\beta}A_\beta = \frac{e_{\alpha\beta}\mathcal{M}_\beta^0}{E h}
\qquad \mbox{and}\qquad
\bar{A}_3 = A_3\bar{s} = - \frac{\mathcal{R}_3^0}{\bar{s}E h} \,,
\end{equation}
where we denote $ \bar{A}_3 = A_3\bar{s} $. From (\ref{eq61}) we can see that $ \bar{A}_3  $ is a global measure of axial strain and $ A_\alpha $ are global measures of axial curvature of the cylindrical shell.

Similarly, the constants $ \hat{B}_i $ can be determined in terms of $ \hat{A}_i $ ($ i=1,2,3 $) from the same relations (\ref{eq1o}). Then, the system of equations (\ref{eq49}) reduces to
\begin{equation}\label{eq64}
I_{\alpha\beta}\,\hat{A}_\beta = -\frac{\mathcal{R}_\alpha^0}{Eh} \qquad \mbox{and}\qquad  \hat{A}_3= 0  \;.
\end{equation}
The constants $ \hat{A}_1\, $, $ \hat{A}_2\, $ can be interpreted as global measures of strain appropriate to flexure.
Due to the fact that $ \boldsymbol B \sim \frac{D}{\bar{s}}\,\boldsymbol A\,  $, we can neglect the terms of the form
\[  
\displaystyle\frac{1}{C}\Big(
s \boldsymbol B\times\boldsymbol e_3 - \int_{0}^{s} \frac{1}{R}(\boldsymbol B\cdot \hat{\boldsymbol r})\boldsymbol{\tau}\,\mathrm{d}s 
\Big) \qquad \mbox{and} \qquad \hat{\boldsymbol B}\cdot \big(\boldsymbol n + \frac{1}{R}\,\hat{\boldsymbol{r}}\big)
\]
in the displacement fields (\ref{eq25}), (\ref{eq47}) and in the relations (\ref{eq54})--(\ref{eq57}). Thus, neglecting the small terms of order $ \big(\frac{h}{\bar{s}}\big)^2\ll 1 $ we find from (\ref{eq54})--(\ref{eq57}) the following expression of the flexure function 
\begin{equation}\label{eq2o}
\begin{array}{c}
\displaystyle
\psi(s) = {K}_0s -\int_{0}^{s}  \Big\{ \hat{A}_\alpha \Big[
\frac{\nu}{2}\, x_\alpha^\prime(x_\beta x_\beta) -\nu x_\alpha(x_\beta x_\beta^\prime)   
+2(1+\nu) \int_{0}^{s}x_\alpha \,\mathrm{d}s \Big]
\vspace{4pt}\\
\qquad\qquad\qquad
\displaystyle
-\, \boldsymbol\tau \cdot\int_{0}^{s} \boldsymbol n \int_{0}^{s} \Big(\frac{\nu}{R}\hat{\boldsymbol A} +  \frac{1}{D}\hat{\boldsymbol B}  \Big)\cdot\hat{\boldsymbol r} \,\mathrm{d}s\,\mathrm{d}s                   
\Big\}\mathrm{d}s ,
\end{array}
\end{equation}
while the constants $ \tilde{K} $ and $ K_0 $ are given by 
\begin{equation}\label{eq3o}
\begin{array}{l}
\displaystyle
\tilde{K} = -\frac{1}{2 \mathcal{A}}\int_{0}^{\bar{s}}  \Big\{ \hat{A}_\alpha \Big[
\frac{\nu}{2}\, x_\alpha^\prime(x_\beta x_\beta) -\nu x_\alpha(x_\beta x_\beta^\prime)   
- \frac{(1+\nu)\bar{s}}{\mathcal{A}}\, x_\alpha \varphi(s)  \Big]
\vspace{4pt}\\
\qquad\qquad\qquad\qquad\qquad\qquad
\displaystyle
-\, \boldsymbol\tau \cdot\int_{0}^{s} \boldsymbol n \int_{0}^{s} \Big(\frac{\nu}{R}\hat{\boldsymbol A} +  \frac{1}{D}\hat{\boldsymbol B}  \Big)\cdot\hat{\boldsymbol r} \,\mathrm{d}s\,\mathrm{d}s                   
\Big\}\mathrm{d}s ,
\vspace{4pt}\\
\displaystyle
K_0 = - \tilde{K}\, \frac{2\mathcal{A}}{\bar{s}}\, - \, \frac{1+\nu}{\mathcal{A}} \int_{0}^{\bar{s}} x_\alpha \int_{0}^{s} \boldsymbol r\cdot \boldsymbol n \,\mathrm{d}s\,\mathrm{d}s .
\end{array}
\end{equation}
Consequently, on the basis of Theorems \ref{Th2} and \ref{Th3} we deduce the following simplified solution to our problem.

\begin{theorem}\label{Th4}
	Consider the relaxed Saint-Venant's problem characterized by the equations (\ref{eq13})--(\ref{eq16}) and the end edge conditions (\ref{eq22}). Then, the approximated form of the solution for the displacement field $ \boldsymbol u = u_i \boldsymbol e_i $ is
\begin{equation}\label{eq4o}	
\begin{array}{l}
	u_\alpha= - \hat{A}_\alpha\big( \frac16\,z^3 - \frac12\,\nu z x_\beta x_\beta\big)-\frac12 A_\alpha(z^2- \nu x_\beta x_\beta) - \nu x_\alpha (z \hat{A}_\beta x_\beta \!+\! {A}_\beta x_\beta \!+\! \bar{A}_3)
		\vspace{4pt}\\
		\displaystyle
	\qquad\;\;	 - (K+\tilde{K}) z e_{\alpha\beta}x_\beta -e_{\alpha\beta} \!\int_{0}^{s}\! x_\beta^\prime\! \int_{0}^{s}\!
\Big[\frac{\nu}{R}(z\hat{\boldsymbol A}+ \boldsymbol A ) +  \frac{1}{D}(z\hat{\boldsymbol B}+ \boldsymbol B)  \Big]\!\cdot\hat{\boldsymbol r} \,\mathrm{d}s\,\mathrm{d}s,
\vspace{4pt}\\
	u_3 = \frac12z^2\big( \hat{A}_\alpha x_\alpha \big) + z\big( {A}_\alpha x_\alpha + \bar{A}_3\big) + (K+\tilde{K}) \varphi(s) + \psi(s),
\end{array}
\end{equation}
where the torsion function $ \varphi(s) $ is given by  (\ref{eq28}), the flexure function $ \psi(s) $ and the constants $ K $,  $ \tilde{K} $, $ \boldsymbol A $, $ \hat{\boldsymbol A} $, $ \boldsymbol B $ and  $ \hat{\boldsymbol B} $  are determined
in terms of the resultant loads $\mathcal{R}_i^0 $ and $ \mathcal{M}_i^0\,$  by the relations (\ref{eq62})--(\ref{eq3o}).
\end{theorem}

\noindent\textbf{Remarks: 1.} In relations (\ref{eq4o}) we recognize the form of classical Saint-Venant's solution for the extension, bending, torsion, and flexure of three-dimensional cylinders, see e.g. \cite{Iesan86,Iesan87}. Except for the integral term in $ u_\alpha\, $, which is specific to shells, the structure of the solution (\ref{eq4o}) coincides with the expression of Saint-Venant's solution.
However, the torsion function for shells (\ref{eq28}) and the relations (\ref{eq62}), (\ref{eq61}), (\ref{eq64}) have different forms as compared to the three-dimensional counterparts. Thus, the solution obtained in this paper are the analogues of the classical Saint-Venant's solutions in the classical theory of shells.

\textbf{2.} The stress state of the shell corresponding to the solution (\ref{eq4o}) is given by
\begin{equation}\label{eq4obis}
\begin{array}{l}
\displaystyle
N_{ss}  = 0,\qquad  \quad
N_{zz}  = C (1-\nu^2)  \big( z \hat{\boldsymbol A} + \boldsymbol A \big)\cdot\hat{\boldsymbol{r}}\,,
\vspace{4pt}\\\displaystyle
N_{sz} = N_{zs} =  C(1-\nu) K\, \frac{\mathcal{A}}{\bar{s}}
- C (1 \!-\! \nu^2)    \hat{\boldsymbol A}  \!\cdot\! \Big( 
\int_{0}^{s} \boldsymbol r\,\mathrm{d}s + \frac{1}{2 \mathcal{A}} \int_{0}^{\bar{s}} \!\boldsymbol r\!\int_{0}^{s}\! \boldsymbol r\cdot\boldsymbol n\,\mathrm{d}s\,\mathrm{d}s
\Big), 
\vspace{4pt}\\
\displaystyle
M_{ss}  = -\big( z \hat{\boldsymbol B} + \boldsymbol B \big)\cdot\hat{\boldsymbol{r}}, 
 \qquad  
M_{zz}  = -D(1-\nu^2)\big( z \hat{\boldsymbol A} + \boldsymbol A \big)\cdot\boldsymbol n - \nu\big( z \hat{\boldsymbol B} + \boldsymbol B \big)\cdot\hat{\boldsymbol{r}}\,,
\vspace{4pt}\\
\displaystyle
M_{sz} = M_{zs} = - D(1-\nu) \Big[K+\tilde{K}-\nu e_{\alpha\beta}\hat{A}_\alpha x_\beta
+ \int_{0}^{s} \Big(\frac{\nu}{R}\hat{\boldsymbol A} +  \frac{1}{D}\hat{\boldsymbol B}  \Big)\cdot\hat{\boldsymbol r} \,\mathrm{d}s \Big].
\end{array}
\end{equation}
In (\ref{eq4obis}) some small terms of order $ \big(\frac{h}{\bar{s}}\big)^2\ll 1 $, such as e.g.  $ {\boldsymbol B}\cdot \big(\boldsymbol n + \frac{1}{R}\hat{\boldsymbol{r}}\big) $ and $ \,\frac{1}{C}   \hat{\boldsymbol B}\cdot  \boldsymbol{\tau}  $, have been neglected, see relations (\ref{eq43}) and (\ref{eq56}) for comparison.
The dependence of the stress tensors on the axial coordinate $ z $ in (\ref{eq4obis}) is at most linear.

\textbf{3.} The torque-twist relation (\ref{eq62}) and the relations (\ref{eq61}), which express the measures of curvature $ A_\alpha $ and stretch $ \bar{A}_3 \,$ in terms of the resultants $\mathcal{M}_\alpha^0 $ and $ \mathcal{R}_3^0\,$,   can be found in various forms in the papers \cite{Reissner59,Reissner72} and in the classical books \cite[Sect. 47]{Sokolnikoff56}
and \cite[Sects. 94, 98, 102]{Timoshenko51}, for some special cases.

\section{Application: Circular cylindrical shells}
\label{Sect6}

The results obtained above are valid for cylindrical tubes with arbitrary cross-sections.  Let us specialize these results to the case of circular cylindrical tubes and derive the solution of the relaxed Saint-Venant's problem. 

The parametrization (\ref{eq9}) of the midsurface has in this case the following form
\begin{equation}\label{eq69}
x_1(s)= R_0\cos \frac{s}{R_0}\;,\qquad x_2(s)= R_0\sin \frac{s}{R_0}\;,\qquad s\in[0,\bar{s}]\,,
\end{equation}
where $ R_0 $ is the radius of the circular cylindrical surface and $ \bar{s}=2\pi R_0\, $.
From (\ref{eq69}) we deduce that
\begin{equation}\label{eq70}
\begin{array}{c}
\displaystyle
x_\alpha^\prime = e_{\beta\alpha}\frac{x_\beta}{R_0}\;,\quad 
x_\alpha^{\prime\prime} = -\frac{x_\alpha}{R_0^2}\;,\quad 
\boldsymbol \tau = \frac{1}{R_0}\,e_{\beta\alpha}x_\beta\boldsymbol e_\alpha\,,\quad
\boldsymbol n = \frac{1}{R_0}\,x_\alpha\boldsymbol e_\alpha\,,
\vspace{6pt}\\
\displaystyle
\int_{0}^{\bar{s}} x_\alpha\,\mathrm{d}s =0,\qquad I_{\alpha\beta}=\int_{0}^{\bar{s}} x_\alpha x_\beta\,\mathrm{d}s= \pi R_0^3 \delta_{\alpha\beta} \,
,\qquad \frac{2\mathcal{A}^2}{\bar{s}}= \pi R_0^3 \,.
\end{array}
\end{equation}
From (\ref{eq1o}) and (\ref{eq70}) we deduce the relations
\begin{equation}\label{eq75}
B_i = - \frac{\nu D}{R_0}\,A_i\qquad \mbox{and}\qquad \hat B_i = - \frac{\nu D}{R_0}\,\hat A_i \quad (i=1,2,3).
\end{equation}
Hence, $ \frac{\nu}{R_0} {\boldsymbol A} +  \frac{1}{D} {\boldsymbol B} = \boldsymbol 0  $  and the integral term in the solution (\ref{eq4o}) vanishes. The relations (\ref{eq62}), (\ref{eq61}) and (\ref{eq64}) reduce to
\begin{equation}\label{eq5o}
K =  \frac{-\mathcal{M}_3^0}{2\pi R_0^3\mu h}, \qquad 
 A_\alpha =  \frac{e_{\alpha\beta}\mathcal{M}_\beta^0}{\pi R_0^3Eh}, \qquad \bar{A}_3 =  \frac{-\mathcal{R}_3^0}{2\pi R_0Eh}, \qquad
 \hat{A}_\alpha = \frac{-\mathcal{R}_\alpha^0}{\pi R_0^3Eh}.
\end{equation}
In view of (\ref{eq70}), the torsion function (\ref{eq28}) is vanishing $ \varphi(s)=0  $ and from 
(\ref{eq2o}) and (\ref{eq3o}) we find
\begin{equation}\label{eq6o}
\psi(s)= \frac{4+3\nu}{2}\, R_0^2\, \hat{A}_\alpha x_\alpha\,, \qquad 
K_0 = 2(1+\nu)R_0^2\, \hat{A}_2 \qquad \mbox{and}\qquad  \tilde{K} = 0.
\end{equation}
Using the above results, we obtain from Theorem \ref{Th4} the following solution
\begin{corollary}\label{Cor1}
	Consider the relaxed Saint-Venant's problem (\ref{eq13})-(\ref{eq16}) with boundary conditions (\ref{eq22}) for circular cylindrical shells. Then, the solution in terms of displacements is given by (up to a rigid body displacement field)
	\begin{equation}\label{eq63}
	\begin{array}{l}
	u_\alpha= - \frac16 z^3 \hat A_\alpha - \frac12 z^2 A_\alpha -\nu z (\hat A_\beta x_\beta)x_\alpha - \nu (A_\beta x_\beta+ \bar{A}_3) x_\alpha  - K z \,e_{\alpha\beta}x_\beta\,,
	\vspace{4pt}\\
	u_3 =  \big[ \frac12 z^2  + 2(1+\nu)R_0^2\, \big] (\hat A_\alpha x_\alpha) + z(A_\alpha x_\alpha+ \bar{A}_3) ,
	\end{array}
	\end{equation}
	where the constants $ K$,  $ A_\alpha \,$, $ \bar{A}_3 $ and $ \hat A_\alpha \,$  represent measures of twist, curvature, stretch and flexure, respectively, which are given in terms of the resultant loads  $\mathcal{R}_i^0 $ and $ \mathcal{M}_i^0 $ by the relations (\ref{eq5o}).
\end{corollary}

These results are consistent with classical solutions for the deformation of circular cylindrical shells, see e.g. \cite{Sokolnikoff56} and \cite{Timoshenko51}.

\section{Final remarks and conclusions}
\label{Sect7}

We have presented a solution procedure to solve the relaxed Saint-Venant's problem for cylindrical Koiter shells. The general method has been established by Ie\c san \cite{Iesan86,Iesan87} in the context of three-dimensional elasticity and applied previously by the author in the case of Cosserat surfaces \cite{Birsan-JE-04}. We have obtained an explicit solution with exact closed form, see Theorems \ref{Th2} and \ref{Th3}. In the case of very thin shells we can neglect some higher order terms and derive an approximate solution, which has a simple form (see Theorem \ref{Th4}). We remark the formal resemblance of the simplified solution with the classical Saint-Venant solutions for solid cylinders. Moreover, the determined solutions possess several characteristical properties of the classical Saint-Venant solution. For instance, the stress tensors $ \boldsymbol N $ and $ \boldsymbol M $ are independent of the axial coordinate $ z $ for the extension-bending-torsion solution (see (\ref{eq43})); for the flexure problem, $ \boldsymbol N $ and $ \boldsymbol M $ depend on $ z $ at most linearly (see (\ref{eq56})).

The solution presented in this paper is new in its generality, although one can find in the literature several approaches to this problem with various solution procedures. Thus, some of the results obtained in Section \ref{Sect5} are well known, such as for instance the formulas (\ref{eq62}), (\ref{eq61}), which express the relations between the measures of twist, curvature and stretch and the twisting moments, bending moments and axial force, respectively. 
The deformation of cylindrical shells was investigated in the monographs of Lurie \cite{Lurie47}, Vlasov \cite{Vlasov49} Timoshenko and Goodier \cite{Timoshenko51}, Sokolnikoff \cite{Sokolnikoff56}, Novozhilov \cite{Novozhilov59}, and Goldenveizer \cite{Goldenveizer61} using different variants of the classical shell theory, also referred to as the first order approximation shell theory. The mostly used variants of the classical shell theory are Love's theory \cite{Love27}, the Novozhilov-Balabuch theory \cite{Novozhilov59} and the Koiter theory \cite{Koiter60}.
All these versions of the classical theory of shells differ from the Koiter model only by small terms of order $ O(\frac{h}{R}) $. For example, Goldenveizer \cite{Goldenveizer61} has employed a shell theory based on the Novozhilov-Balabuch equations to investigate the deformation of cylindrical shells: the equilibrium equations were integrated by using trigonometric series expansions (for circular cylindrical shells) or Fourier series expansions (for general cylindrical shells), when the end edge are either free, hinged or clamped. 

In our paper, we have used the Koiter shell model, since it is a consistent first order model which
has the advantage of simplicity as compared to other more refined models and has a wide range of applicability \cite{Ventsel-01}. Classical shell theory works very well if it is applied to problems which fulfill all the constraints posed by the accepted assumptions. These assumptions include the Kirchhoff hypotheses (i.e., normals to the undeformed middle surface remain straight and normal to the deformed middle surface and undergo no extension; moreover, the transverse normal stress is small compared with other normal stress components and may be neglected), together with the small-deflection assumption and the thinness assumption. This set of assumptions is usually called the Kirchhoff--Love hypotheses and is a common feature of any first-order approximation shell theory. For such models, the extension and bending deformations of isotropic shells are decoupled.

It is known that Kirchhoff--Love hypotheses introduce in the classical shell equations errors of order $ O(\frac{h}{R}) $ in comparison with 1. (In general, shell theories neglecting terms of order $ O(\frac{h}{R}) $ are called classical.)  Several authors have shown in \cite{Morgenstern59,John65,Koiter71,Koiter-Sim73,Ladeveze76} that the classical shell theory gives correct predictions for the stress state of a shell problem, i.e. with small errors of order $ O(\frac{h}{R}) $ 
as compared to the verified solution, given by improved shell models of higher accuracy level.
Accordingly, our solution for the stress state (\ref{eq4obis}) to the relaxed Saint-Venant's problem is also correct in the first approximation sense. 

On the other hand, the predictions for the displacement field of the classical shell theory may be incorrect under certain conditions, as it was pointed out in the paper \cite{Berdic-Mis92}. 
This effect of accuracy loss for displacements in cylindrical shells is only possible when the shell extension is strong, i.e. when (see relation (14) in \cite{Berdic-Mis92})
\begin{equation}\label{eq8o}
\max_{\mathcal{S}} \| \boldsymbol \epsilon \| \,\,\gg\,\, h\, \max_{\mathcal{S}} \| \boldsymbol \rho \| \;,
\end{equation}
where $\, \| \boldsymbol \epsilon \|^2 = \epsilon_{\alpha\beta} \epsilon^{\alpha\beta}\, $ and \,$ \| \boldsymbol \rho \|^2 =  \rho_{\alpha\beta} \rho^{\alpha\beta} \, $ are the norms of the strain tensors. In case of accuracy loss for displacements, one should solve the problem using a more refined shell model. Refined shell theories provide two-dimensional shell models of higher accuracy level, in which the coupling between extension and bending deformations is taken into account. This means, for instance, that extensional forces are present in a shell subjected to bending, and also, that bending moments appear in a shell under extension.
For a comparison of the most important variants of the classical refined shell theories, written in unified notations, we refer to the paper of Berdichevsky and Misyura \cite{Berdic-Mis92}.
 Since the condition (\ref{eq8o}) is fulfilled for the solution given by Theorem \ref{Th4}, we verify the accuracy of the displacement solution (\ref{eq4o}) by comparison with other results established previously in the framework of refined shell theories for some special cases.
Reissner and Tsai \cite{Reissner72} have used a six-parameter shell model to solve the extension-bending-torsion problem, assuming that stress and strain tensors are independent of the axial coordinate $ z $. In case of isotropic shells, the displacement solution for the extension-bending-torsion problem is in accordance with our results in Theorem \ref{Th4} and Corollary \ref{Cor1}.
The same solution can be obtained also from the work of Berdichevsky et al. \cite{Berdichevsky92}; see also Ladev\`eze et al. \cite{Ladeveze04}. The Saint-Venant's problem for shells has been approached previously by the author using other refined models, such as the theory of Cosserat shells \cite{Birsan-JE-04}, or the theory of directed surfaces of Zhilin and Altenbach \cite{Altenbach-Zhilin-88,Altenbach04,Zhilin76,Zhilin06} in \cite{Birsan-Alten-2011,Birsan-JTS-13}. One can see that the displacement fields obtained in these different approaches compare very well with solution (\ref{eq4o}), which shows that it does not feature any accuracy loss.

In conclusion, the presented solutions are the counterparts of Saint-Venant's solutions in the classical theory of shells. These can be used in applications and as benchmark solutions for the comparison with numerical results. Although some of these results are known for certain special cases, we present them here in an unified manner, together with a general solution procedure. Finally, we mention that this method can be extended to solve the relaxed Saint-Venant's problem for orthotropic or anisotropic shells \cite{Birsan-IJES-09,Birsan-Alten-2011} 
or to investigate thermal stresses in cylindrical shells \cite{Birsan-EJM/S-09,Birsan-JTS-13}.

\bigskip\bigskip\bigskip

\textbf{Acknowledgements:}
This research has been funded by the Deutsche Forschungsgemeinschaft (DFG, German Research Foundation; project no. 415894848).


\bibliographystyle{plain} 

\bibliography{literatur_Birsan}

\end{document}